\documentclass[reqno,11pt]{amsart}
\usepackage{amssymb}
\usepackage{verbatim}
\usepackage[hypertex]{hyperref}

\newtheorem{theorem}{Theorem}[section]
\newtheorem{corollary}[theorem]{Corollary}
\newtheorem{lemma}[theorem]{Lemma}
\newtheorem{proposition}[theorem]{Proposition}

\theoremstyle{definition}
\newtheorem{remark}[theorem]{Remark}

\theoremstyle{definition}

\theoremstyle{definition}

\def\bR{\mathbb{R}}

\def\bZ{\mathbb{Z}}

\def\cF{\mathcal{F}}

\DeclareMathOperator{\osc}{osc}

\makeatletter
\def\dashint{\operatorname%
{\,\,\text{\bf--}\kern-.98em\DOTSI\intop\ilimits@\!\!}}
\makeatother

\newcommand{\mysection}[1]{\section{#1}
\setcounter{equation}{0}}

\begin{document}
\title[Non-local elliptic equations]{Schauder estimates for a class of non-local elliptic equations}

\author[H. Dong]{Hongjie Dong}
\address[H. Dong]{Division of Applied Mathematics, Brown University,
182 George Street, Providence, RI 02912, USA}
\email{Hongjie\_Dong@brown.edu}
\thanks{H. Dong was partially supported by the NSF under agreement DMS-0800129 and DMS-1056737.}

\author[D. Kim]{Doyoon Kim}
\address[D. Kim]{Department of Applied Mathematics, Kyung Hee University, 1732 Deogyeong-daero, Giheung-gu, Yongin-si, Gyeonggi-do 446-701, Republic of Korea}
\email{doyoonkim@khu.ac.kr}
\thanks{D. Kim was supported by Basic Science Research Program through the National Research Foundation of Korea (NRF) funded by the Ministry of Education, Science and Technology (2011-0013960).}

\subjclass[2010]{45K05,35B65,60J75}

\keywords{non-local elliptic equations, Schauder estimates, L\'evy processes.}

\begin{abstract}
We prove Schauder estimates for a class of non-local elliptic operators with kernel $K(y)=a(y)/|y|^{d+\sigma}$ and either Dini or H\"older continuous data. Here $0 < \sigma < 2$ is a constant and $a$ is a bounded measurable function, which is not necessarily to be homogeneous, regular, or symmetric. As an application, we prove that the operators give isomorphisms between the Lipschitz--Zygmund spaces $\Lambda^{\alpha+\sigma}$ and $\Lambda^\alpha$ for any $\alpha>0$. Several local estimates and an extension to operators with kernels $K(x,y)$ are also discussed.
\end{abstract}

\maketitle

\mysection{Introduction and Main results}
The objective of this paper is to prove Schauder estimates  for a class of non-local elliptic equations. There is a vast literature on Schauder estimates for second-order elliptic and parabolic equations; see, for instance, \cite{GT, Lieb96} and references therein. If $L=a^{ij}D_{ij}$ is a second-order uniformly elliptic operator with constant coefficients, then the classical Schauder estimate for $L$ in the whole space $\bR^d$ is of the form:
\begin{equation}
                                        \label{eq11.47}
[D^2 u]_{C^\alpha(\bR^d)}\le N[Lu]_{C^\alpha(\bR^d)},
\end{equation}
where $\alpha\in (0,1)$, $[\,\cdot\,]_{C^\alpha(\bR^d)}$ stands for the usual H\"older semi-norm in the whole space, and $N$ is a constant depending only on $d$, $\alpha$, and the ellipticity constant of the coefficients $a^{ij}$. In general, the estimate \eqref{eq11.47} does not hold when $\alpha=0$ or $1$. However, if $Lu$ is Dini continuous (see the definition below), then $D^2 u$ is uniformly continuous. See, for instance, \cite{GT}.

We are interested in this type of estimates for a non-local elliptic equation associated with pure jump L\'evy processes:
\begin{equation}
                                \label{elliptic}
Lu-\lambda u=f
\end{equation}
in the whole space or on a bounded domain.
Here
\begin{equation}
                                    \label{eq23.22.24}
L u (x) = \int_{\bR^d} \left(u(x+y) - u(x)-y\cdot \nabla u(x)\chi^{(\sigma)}(y)\right)K(x,y)\, d y,
\end{equation}
$\lambda\ge 0$ is a constant, $\sigma \in (0, 2)$, $\chi^{(\sigma)}$ is a suitable indicator function, and $K$ is a nonnegative kernel. In the most part of the paper we focus our efforts on the case when $K(x,y)=K(y)$ and the equation is satisfied in the whole space $\bR^d$.
This type of non-local operators naturally arises from models in physics,
finance, and engineering that involve long-range interactions, and has
attracted the attention of many mathematicians. A model case of such operators is the fractional Laplace operator $(-\Delta)^{\sigma/2}$, which has the symbol $|\xi|^\sigma$.
In this case and, in general, if the symbol of the operator is sufficiently smooth and its derivatives satisfy appropriate decays at infinity, Schauder estimates like \eqref{eq11.47} follow directly from the classical theory for pseudo-differential operators, which can be found, for instance, in \cite{St70,St93,Ja01}.
In particular, the boundedness of related pseudo-differential operators in H\"{o}lder spaces is discussed in \cite{St70, St93}.
We also mention that various regularity issues of solutions to non-local elliptic equations, such as the Harnack inequality, H\"{o}lder estimates, and non-local versions of the Aleksandrov--Bakelman--Pucci estimate, were studied first by using probabilistic methods, and more recently, by analytic methods; see, for instance, \cite{Bas02, Bas05_1, Bas05_2} and \cite{CS09, Ka09,Si06, Ba11, CS11, KL10}.

In \cite{DK11_01} by using a purely analytic method, we established $L_p$-estimates for \eqref{elliptic} under the following assumptions on $L$. The nonnegative kernel $K$ is translation invariant with respect to $x$ and satisfies the lower and upper bounds
\begin{equation}
								\label{eq1211}	
(2-\sigma)\frac{\nu}{|y|^{d+\sigma}}
\le K(y) \le (2-\sigma)\frac{\Lambda}{|y|^{d+\sigma}}
\end{equation}
for any $y\in \bR^d$, where $0<\nu\le \Lambda<\infty$ are two constants. The bounds in \eqref{eq1211} can be viewed as an ellipticity condition on the operator $L$. We take
$$
\chi^{(\sigma)}\equiv 0\quad\text{for}\,\,\sigma\in (0,1),
\quad \chi^{(1)}=1_{y\in B_1},\quad
\chi^{(\sigma)}\equiv 1\quad\text{for}\,\,\sigma\in (1,2).
$$
If $\sigma=1$, a cancellation condition is imposed on $K$:
\begin{equation}
                                    \label{eq21.47}
\int_{\partial B_r}y K \, dS_r(y)=
0,\quad \forall r\in (0,\infty),
\end{equation}
where $dS_r$ is the surface measure on $\partial B_r$. Thus the indicator function $\chi^{(1)}$ can be replaced by $1_{B_r}$ for any $r>0$.
We remark that \eqref{eq21.47} is a quite standard condition which is also used, for example, in \cite{MP92}, and it is always satisfied by any symmetric kernels $K(-y)=K(y)$. In \cite{AK09} the principal part of the kernel is assumed to be symmetric when $\sigma=1$, so that, loosely speaking, \eqref{eq21.47} is satisfied. Note that we do not require $K$ to be either homogeneous, regular, or symmetric.
The symbol of $L$ is given by
\begin{equation}
								\label{eq1102}
m(\xi) = \int_{\bR^d} \left(e^{iy\cdot\xi } - 1-iy\cdot \xi\chi^{(\sigma)}(y) \right) K(y) \, dy,
\end{equation}
which generally lacks sufficient differentiability.
Therefore, the classical theory for pseudo-differential operators or the Calder\'on--Zygmund approach is no longer applicable. In this paper, we prove Schauder estimates for \eqref{elliptic} under the very same conditions as in \cite{DK11_01}.
Furthermore, while in \cite{DK11_01} we dealt with $K$ independent of $x$, we prove the same estimates for operators with $x$-dependent kernels under reasonable assumptions on $a(x,y)$.
In fact, the $x$-dependent kernel case follows from a priori Schauder estimates for $x$-independent kernels and the boundedness of corresponding operators, the latter of which is easier than that  in $L_p$-theory.

Perhaps the closest papers to the subject of the present one are Mikulevicius and Pragarauskas \cite{MP92} and Bass \cite{Bas09}. In \cite{MP92} the well-posedness of the Cauchy problem for the operator $L$ in \eqref{eq23.22.24} was obtained in both $L_p$ spaces and H\"older spaces. As the proofs in \cite{MP92} are based on the Calder\'on--Zygmund approach, the kernel $K(x,y)$ is assumed to be homogeneous and sufficiently smooth with respect to $y$, and its derivatives in $y$ to be H\"older continuous with respect to $x$. We also refer the reader to \cite{AK09}, where a similar well-posedness result was proved under the condition that the principal part of the kernel is smooth in $y$ and H\"older in $x$. By using a combination of probabilistic and analytic methods, in \cite{Bas09} Bass established Schauder estimates for non-local operators similar to \eqref{elliptic} under the conditions that $\sigma+\alpha$ is a non-integer and $K(x,y)$ is H\"older continuous in $x$ satisfying \eqref{eq1211}. Compared to \cite{Bas09}, our approach is purely analytic; see also Remark \ref{rm1.33}. Moreover, our results do not require the non-integer condition on $\sigma+\alpha$, and we also deal with equations with Dini continuous data and solutions in more general Lipschitz--Zygmund spaces.

Once this paper has been completed, we learned that in a very recent preprint \cite{MP11} Mikulevicius and Pragarauskas obtained similar solvability results in H\"older spaces for non-local parabolic equations under a slightly weaker ellipticity condition. The proofs in \cite{MP11} are mainly based on a probabilistic argument and their previous results in \cite{MP92}.

Before we state the main result, we fix a few notation. We use $H_2^{\sigma}(\bR^d)$ to denote the Bessel potential space: $H_2^{\sigma}(\bR^d)=(1-\Delta)^{-\sigma/2}L_2(\bR^d)$. Throughout the paper we omit $\bR^d$ in $C^\alpha(\bR^d)$, $C_0^\infty(\bR^d)$, or $H_2^\sigma(\bR^d)$, etc. whenever the omission is clear from the context.
We write $N(d,\nu,...)$ in the estimates to express that the constant $N$ depends only on the parameters $d, \nu, ...$.

A bounded continuous increasing function $\omega$  on $\overline{\bR_+}$ is called a Dini function if $\omega(0)=0$ and
$$
\int_0^1 \omega(s)/s\,ds<\infty.
$$
For a function $f$ in $\bR^d$, we define its modulus of continuity $\omega_{f}$ by
$$
\omega_f(r)=\sup_{x,y\in\bR^d,|x-y|\le r}|f(x)-f(y)|.
$$
We say that a bounded function $f$ in $\bR^d$ is Dini continuous if its modulus of continuity $\omega_{f}$ is a Dini function. In this case, we write $u\in C^{\text{Dini}}$. Clearly, any function in $C^{\alpha}$ is Dini continuous.

Now we state the main results of the paper.

\begin{theorem}
 								\label{mainth00}
Let $0< \sigma < 2$, $\lambda\ge 0$, and $u\in C^2_{\text{loc}} \cap L_\infty$ be a solution to \eqref{elliptic} in $\bR^d$. Assume that $K=K(y)$ satisfies \eqref{eq1211} and, if $\sigma=1$, $K$ also satisfies the condition \eqref{eq21.47}.

i) Suppose $f\in L_\infty$. Then we have
\begin{equation}
                                \label{eq07.09.53}
\lambda \|u\|_{L_\infty}
\le \|f\|_{L_\infty}.
\end{equation}

ii) Suppose $f\in C^{\text{Dini}}$. Then $(-\Delta)^{\sigma/2}u$ is bounded and uniformly continuous. Moreover, the modulus of continuity of $(-\Delta)^{\sigma/2}u$ is controlled by $\omega_f$, and independent of $\lambda$ and $\sigma$ as $\sigma\to 2$.
\end{theorem}

We have more precise estimates when $f$ is H\"older continuous.

\begin{theorem}[Schauder estimate]
								\label{mainth01}
Let $0< \sigma < 2$, $\lambda\ge 0$, and $u\in C^2_{\text{loc}} \cap L_\infty$ be a solution to \eqref{elliptic} in $\bR^d$. Assume that $K=K(y)$ satisfies \eqref{eq1211} and, if $\sigma=1$, $K$ also satisfies the condition \eqref{eq21.47}.

i) Suppose $\alpha\in (0,1)$ and $f\in C^\alpha$. Then we have $u, (-\Delta)^{\sigma/2}u\in C^\alpha$ and the following estimate holds:
\begin{equation}
                                \label{eq24.09.53}
[(-\Delta)^{\sigma/2}u]_{C^\alpha}+\lambda [u]_{C^\alpha}
\le N [f]_{C^\alpha},
\end{equation}
where $N=N(d,\nu,\Lambda,\sigma,\alpha)$.

ii) Suppose $f\in C^{0,1}$. Then for any $x,y\in \bR^d$ satisfying $|x-y|\le 1/2$ we have
\begin{multline}
                                \label{eq24.09.53z}
|(-\Delta)^{\sigma/2}u(x)-(-\Delta)^{\sigma/2}u(y)|+\lambda|u(x)-u(y)|\\
\le N \big|(x-y)\log|x-y|\big|\|f\|_{C^{0,1}},
\end{multline}
where $N=N(d,\nu,\Lambda,\sigma)$ is uniformly bounded as $\sigma\to 2$.
\end{theorem}

We note that under the conditions of Theorem \ref{mainth01} ii), $(-\Delta)^{\sigma/2}u$ is actually in the Zygmund space $\Lambda^1$. Moreover, in this case the condition $f\in C^{0,1}$ can be relaxed to $f\in \Lambda^1$; see Section \ref{sec4}.

Our proof is purely analytic and based on Campanato's approach in a novel way. The main step of Campanato's approach is to show that the mean oscillations
of the derivatives of $u$ in balls vanish in certain order as the radii of balls go to zero. Unlike the second-order case, when estimating the mean oscillations for a solution to the non-local equation \eqref{elliptic}, we need to take care of the contribution of the solution from the entire space $\bR^d$. For this purpose, the key estimate in our argument is Proposition \ref{cor1}, in which we bound the H\"older semi-norm of the solution by its mean oscillations in a sequence of expanding balls.

\begin{remark}
                \label{rm0}
In Theorem \ref{mainth01}, the ellipticity condition \eqref{eq1211} can be
relaxed. For instance, we may allow $K$ to vanish for any $y\in B_1^c$, at the cost that $\lambda$ should be assumed to be strictly positive and the constants $N$ in \eqref{eq24.09.53} and \eqref{eq24.09.53z} depend also on $\lambda$. See the remark at the end of Section \ref{sec3} for a proof.
\end{remark}

\begin{remark}      \label{rm1.33}
Estimates similar to \eqref{eq24.09.53} were obtained in \cite{Bas09} with the restriction that $\sigma+\alpha$ is not an integer. In the case $\sigma\in (1,2)$, the operators studied in \cite{Bas09} are slightly different from those considered in the current paper. The proofs there use a combination of probabilistic and analytic arguments and, in particular, use certain properties of semigroups generated by stable-like processes. It is however not clear to us if the methods in \cite{Bas09} can be exploited to treat the case when $f$ is only Dini continuous. We also remark that although the cancellation condition \eqref{eq21.47} is omitted in \cite{Bas09}, it is actually implicitly used in the proof of Corollary 5.2 there.
\end{remark}

\begin{remark}
In the borderline case $\alpha=1$, it is well known that \eqref{eq24.09.53} does not hold even for the second-order Poisson equation. An estimate similar to \eqref{eq24.09.53z} for the second-order elliptic equations can be found in \cite{Wa06}. Our treatment of this case is rather elementary.
\end{remark}


Next we state the following solvability result, which implies that $L-\lambda$ is an isomorphism from $\Lambda^{\alpha+\sigma}$ to $\Lambda^{\alpha}$ for any $\alpha>0$. See Section \ref{sec4} for the definition of the Lipschitz--Zygmund spaces $\Lambda^\alpha$.
\begin{theorem}[Solvability in Lipschitz--Zygmund spaces]
                        \label{thm2}
Let $0< \sigma < 2$, $\alpha>0$, and $\lambda>0$. Assume that $K$ satisfies \eqref{eq1211} and, in the case $\sigma=1$, it also satisfies the cancellation condition \eqref{eq21.47}.

i) Then $L-\lambda$ is a continuous operator from $\Lambda^{\alpha+\sigma}$ to $\Lambda^{\alpha}$, i.e., for any $u\in \Lambda^{\alpha+\sigma}$ we have
\begin{equation}
                                            \label{eq22.100}
\|(L-\lambda) u\|_{\Lambda^{\alpha}}\le N\|u\|_{\Lambda^{\alpha+\sigma}},
\end{equation}
where $N=N(d,\Lambda,\sigma,\alpha,\lambda)>0$ is a constant.
Moreover, for any $u\in \Lambda^{\alpha+\sigma}$,
\begin{equation}
                                    \label{eq17.24}
\|u\|_{\Lambda^{\alpha+\sigma}}\le N\|Lu-\lambda u\|_{\Lambda^{\alpha}},
\end{equation}
where $N=N(d,\nu,\Lambda,\sigma,\alpha,\lambda)>0$.

ii) For any $f\in \Lambda^\alpha$, there is a unique solution $u\in \Lambda^{\alpha+\sigma}$ of the equation \eqref{elliptic} in $\bR^d$ satisfying \eqref{eq07.09.53}.
Furthermore, $u$ satisfies \eqref{eq24.09.53} if $\alpha\in (0,1)$, and satisfies
\begin{equation}
										\label{eq0316_01}
|(-\Delta)^{\sigma/2}u(x)-(-\Delta)^{\sigma/2}u(y)|+\lambda|u(x)-u(y)|\\
\le N \big|(x-y)\log|x-y|\big|\|f\|_{\Lambda^1}	
\end{equation}
for any $|x-y|\le 1/2$ if $\alpha=1$.

The constants $N$ are uniformly bounded as $\sigma\to 2$.
\end{theorem}

To deal with $x$-dependent kernels, we derive the following local estimates from Theorem \ref{mainth01} by using a more or less standard localization argument.
\begin{corollary}
                                    \label{thm3}
Let $\sigma\in (0,2)$ and $\alpha\in (0,1)$ be constants. Suppose that $K$ satisfies the same assumptions as in Theorem \ref{thm2} and $u \in C^2(\bR^d)$ satisfies $L u -\lambda u= f$ in $B_3$ for $f \in C^\alpha(B_3)$ and some constant $\lambda\ge 0$.
Then we have
\begin{equation}
                                \label{eq21.06}
[(-\Delta)^{\sigma/2}u]_{C^\alpha(B_1)}\le
N\|f\|_{C^\alpha(B_3)}+N\|u\|_{C^\alpha(\bR^d)}
\end{equation}
for $\sigma\in (0,1)$,
\begin{equation}
                                \label{eq21.06b}
[(-\Delta)^{\sigma/2}u]_{C^\alpha(B_1)}\le
N\|f\|_{C^\alpha(B_3)}+N\|u\|_{C^{1+\alpha}(\bR^d)}
\end{equation}
for $\sigma\in (1,2)$, and
\begin{equation}
                                \label{eq21.06c}
[(-\Delta)^{\sigma/2}u]_{C^\alpha(B_1)}\le N\|f\|_{C^\alpha(B_3)}
+ N(\varepsilon)\|u\|_{L_\infty(\bR^d)}+\varepsilon\|u\|_{C^{1+\alpha}(\bR^d)}
\end{equation}
for $\sigma=1$ and any $\varepsilon\in (0,1)$.
\end{corollary}

Our last result is about the solvability for non-local operators with $x$-dependent kernels.

\begin{theorem}
                                \label{thm4}
Let $\sigma\in (0,2)$, $\alpha\in (0,1)$, and $\lambda>0$ be constants. Assume that $K(x,y)=a(x,y)|y|^{-d-\sigma}$, where $a(\cdot,y)$ is a $C^\alpha$ function for any $y\in \bR^d$ with a uniform $C^\alpha$ norm. In the case $\sigma=1$, the condition \eqref{eq21.47} is also satisfied for any $x\in \bR^d$. Then $L$ defined in \eqref{eq23.22.24} is a continuous operator from $\Lambda^{\alpha+\sigma}$ to $C^{\alpha}$. Moreover, for any $f\in C^\alpha$, there is a unique solution $u\in \Lambda^{\alpha+\sigma}$ of the equation \eqref{elliptic} in $\bR^d$ satisfying
$$
\|u\|_{\Lambda^{\sigma+\alpha}}\le N\|f\|_{C^\alpha},
$$
where $N=N(d,\Lambda,\sigma,\alpha,\lambda,\sup_y\|a(\cdot,y)\|_{C^\alpha})>0$
is independent of $\lambda$ if $\lambda\ge 1$ and uniformly bounded as $\sigma\to 2$.
\end{theorem}

We prove Theorems \ref{mainth00} and \ref{mainth01}, in Section \ref{sec3} after discussing some auxiliary results in Section \ref{sec2}.
As an application of Theorem \ref{mainth01}, in Section \ref{sec4} we derive the unique solvability of \eqref{elliptic} in the Lipschitz--Zygmund spaces $\Lambda^\alpha$, i.e., Theorem \ref{thm2}, which together with a continuity estimate implies that the operator $L-\lambda$ gives an isomorphism between $\Lambda^{\alpha+\sigma}$ and $\Lambda^\alpha$ for any $\alpha>0$.
Section \ref{sec5} is devoted to the proofs of Corollary \ref{thm3} and Theorem \ref{thm4}.

\mysection{Auxiliary estimates}							\label{sec2}

The objective of this section is to present several auxiliary estimates, which will be used in the proofs of the main theorems.
For the proof of Theorem \ref{mainth00}, we need the following properties of Dini functions.
\begin{lemma}
                                                    \label{lem4.08}
Suppose that $\omega$ is a Dini function, and
\begin{equation}
                                    \label{eq4.22}
\tilde \omega(t):=\sum_{k=0}^\infty a^k \omega(b^k t)
\end{equation}
for some constants $a\in (0,1)$ and $b>1$.
Then $\tilde \omega$ is also a Dini function.
\end{lemma}
\begin{proof}
First we note that a bounded continuous increasing function on $\overline{\bR_+}$ must be uniformly continuous on $\overline{\bR_+}$. Because $a\in (0,1)$ and $b>1$, $\tilde \omega$ is a bounded continuous increasing function and $\tilde \omega(0)=0$. Let $t\in (0,1]$ and $k_0=[-\log t/\log b]$.
It follows from \eqref{eq4.22} that
$$
\tilde \omega(t)=\left(\sum_{k=0}^{k_0}+\sum_{k_0+1}^\infty\right)a^k\omega(b^k t)\le \sum_{k=0}^{k_0} a^k\omega(b^k t)+Nt^{\gamma}
$$
for some constant $N$ depending only on $a,b$, and $\sup \omega$, and $\gamma>0$ depending only on $a$ and $b$. By Fubini's theorem, we also get $\int_0^1 \tilde \omega(t)/t\,dt<\infty$. The lemma is proved.
\end{proof}

Throughout the paper we denote
\begin{equation*}
(f)_{\Omega} = \frac{1}{|\Omega|} \int_{\Omega} f(x) \, dx
= \dashint_{\Omega} f(x) \, dx,
\end{equation*}
where $|\Omega|$ is the
$d$-dimensional Lebesgue measure of $\Omega$.

\begin{lemma}
                                                \label{lem4.09}
Let  $\omega$ be a Dini function and $f$ be a bounded continuous function in $\bR^d$. Suppose that for any $x_0\in \bR$ and $l\in \bZ$ we have
$$
\dashint_{B_{2^{l}}(x_0)}\big|f-(f)_{B_{2^{l}}(x_0)}\big|\,dx\le \omega(2^l).
$$
Then for any $x,y\in\bR^d$,
\begin{equation}
                                    \label{eq11.27}
|f(x)-f(y)|\le N\int_0^{|x-y|}\omega(s)/s\,ds+N\omega(2|x-y|),
\end{equation}
where $N>0$ depends only on $d$. In particular, $f$ is a uniformly continuous function in $\bR^d$.
\end{lemma}

\begin{proof}
We follow an idea in \cite[Sect. 6 (iii)]{Sperner}, where the constant $N$ also depends on $\|f\|_{L_\infty}$ because a slightly different condition is imposed there. We give a detailed proof for the sake of completeness.

For any given different $x_0,y_0\in \bR^d$, denote $z_0=(x_0+y_0)/2$ and $r=|x_0-y_0|$. Let $k$ be the integer such that
\begin{equation}
								\label{eq0302_3z}
2^{k+1} \le r < 2^{k+2}.
\end{equation}
By the triangle inequality,
\begin{align}
                                    \label{eq11.47z}
|&f(x_0)-f(y_0)|\nonumber\\
&\le |f(x_0)-(f)_{B_{2^k}(x_0)}|+|f(y_0)-(f)_{B_{2^k}(y_0)}|+|(f)_{B_{2^k}(x_0)}-(f)_{B_{2^k}(y_0)}|\nonumber\\
&:=I_1+I_2+I_3.
\end{align}
First we bound $I_3$. Observe that from \eqref{eq0302_3z}, $B_{2^k}(x_0)\cup B_{2^k}(y_0)\subset B_{ 2^{k+2}}(z_0)$. By the triangle inequality,
\begin{align}
I_3&\le \dashint_{B_{2^k}(x_0)}\dashint_{B_{2^k}(y_0)}|f(x)-f(y)|\,dx\,dy\nonumber\\
&\le N \dashint_{B_{ 2^{k+2}}(z_0)}\dashint_{B_{2^{k+2}}(z_0)}|f(x)-f(y)|\,dx\,dy\nonumber\\
&\le N \dashint_{B_{2^{k+2}}(z_0)}|f(x)-(f)_{B_{2^{k+2}}(z_0)}|\,dx\nonumber\\
                                        \label{eq12.05}
&\le N\omega(2^{k+2})
\le N \omega(2r).
\end{align}
Next, because of the continuity of $f$, we have
$$
f(x_0)=\lim_{l\to -\infty}(f)_{B_{2^l}(x_0)},
$$
which together with the triangle inequality yields
\begin{align}
I_1&\le \sum_{l=-\infty}^k \big|(f)_{B_{2^{l-1}}(x_0)}-(f)_{B_{2^{l}}(x_0)}\big|\nonumber\\
                                    \nonumber
&\le N\sum_{l=-\infty}^k\omega(2^l)\le N\int_0^{r}\omega(s)/s\,ds.
\end{align}
Similarly, we have
\begin{equation}
                                    \label{eq12.18}
I_2\le N\int_0^{r}\omega(s)/s\,ds.
\end{equation}
Combining \eqref{eq11.47z} with \eqref{eq12.05}-\eqref{eq12.18} gives \eqref{eq11.27}. The lemma is proved.
\end{proof}

We will use the following elementary estimates involving H\"older semi-norms.

\begin{lemma}
                                        \label{lem1}
Let $\alpha\in (0,1]$, $\beta\in (0,\alpha)$, $h>0$, and $h\in\bR^d$ be a nonzero vector. Suppose that $f\in C^{\alpha}$ is a function in $\bR^d$. Define
$$
f_{\beta,h}(x)=|h|^{-\beta}(f(x+h)-f(x)).
$$
Then we have $f_{\beta,h}\in C^{\alpha-\beta}$ and
\begin{equation}
								\label{eq0302_4}
[f_{\beta,h}]_{C^{\alpha-\beta}}\le 2[f]_{C^\alpha}.	
\end{equation}
\end{lemma}
\begin{proof}

Note that
\begin{align*}
\left|f_{\beta,h}(x)-f_{\beta,h}(y)\right|
&\le |h|^{-\beta}\left|f(x+h)-f(y+h)\right|
+ |h|^{-\beta}\left|f(x)-f(y)\right|\\
&\le 2 |h|^{-\beta}|x-y|^\beta|x-y|^{\alpha-\beta} [f]_{C^{\alpha}}.
\end{align*}
Thus the inequality \eqref{eq0302_4} holds true if $|x-y|\le h$.
In case $|x-y| > h$, we proceed
\begin{align*}
\left|f_{\beta,h}(x)-f_{\beta,h}(y)\right|
&\le |h|^{-\beta}\left|f(x+h)-f(x)\right|
+ |h|^{-\beta}\left|f(y+h)-f(y)\right|\\
&=\left|\frac{f(x+h)-f(x)}{h^\alpha}\right|h^{\alpha-\beta}
+ \left|\frac{f(y+h)-f(y)}{h^\alpha}\right|h^{\alpha-\beta}\\
&\le 2 |x-y|^{\alpha-\beta} [f]_{C^{\alpha}}.
\end{align*}
Thus we again arrive at the inequality \eqref{eq0302_4}.
\end{proof}

\begin{lemma}
                                        \label{lem2}
Let $\alpha\in (0,1)$, $\beta \in (0,1)$, $K>0$ be constants, and $f$ be a bounded
function on $\bR$. For $h>0$, define
$$
f_{\beta,h}(x)=h^{-\beta}(f(x+h)-f(x)).
$$
Suppose that for any $h>0$, $f_{\beta,h}\in C^\alpha$ and $[f_{\beta,h}]_{C^\alpha}\le K$. Then

i) If $\alpha+\beta<1$, then we have $f\in C^{\alpha+\beta}$ and
$[f]_{C^{\alpha+\beta}}\le NK$, where $N$ depends only on $\alpha+\beta$.

ii) If $\alpha+\beta=1$, then for any $x,y\in \bR$ we have
\begin{equation}
								\label{eq0302_2}
|f(x+y)-f(x)|\le N\big|y\log|y|\big|K+N|y|\min\{\|f\|_{L_\infty},[f]_{C^{\gamma}}\},	
\end{equation}
where $\gamma\in (0,1)$ is arbitrary and $N=N(\gamma)>0$.
\end{lemma}
\begin{proof}

We follow the steps in the proof of \cite[Lemma 5.6]{CaCa95}. See also \cite[Exercise 3.3.3]{Kr96}.
First, we consider the case $\alpha+\beta<1$. In this case we prove that
\begin{equation}
								\label{eq0301}
|f(x+y) - f(x)| \le N(\alpha,\beta) K y^{\alpha+\beta}
\end{equation}
for any $y > 0$.
Let $K_0 = \sup|f|$ and find a positive integer $k$ such that
\begin{equation}
								\label{eq0302_1}
2^{-k}K_0 \le K y^{\alpha+\beta}.	
\end{equation}
Set $\tau_0 = 2^k y$ and define
$$
w(\tau) = f(x+\tau) - f(x).
$$
Note that
\begin{align*}
&|w(\tau) - 2 w(\tau/2)|
= |f(x+\tau) - 2 f(x+\tau/2) + f(x)|\\
&\,= |\left(f(x+\tau) - f(x+\tau/2)\right) - \left(f(x+\tau/2) - f(x)\right)|\\
&\,=\left(\tau/2\right)^\beta
\left|f_{\beta,\tau/2}(x+\tau/2) - f_{\beta,\tau/2}(x)\right|
\le K \left(\tau/2\right)^{\alpha+\beta}.
\end{align*}
Thus we have
$$
\left| 2^{j-1}w(\tau_0/2^{j-1}) - 2^j w(\tau_0/2^j) \right|
\le K 2^{j-1} \left(\tau_0/2^j\right)^{\alpha+\beta}
= K2^{j\left(1-(\alpha+\beta)\right)-1} \tau_0^{\alpha+\beta},
$$
which implies
\begin{multline*}
\left|w(\tau_0) - 2^k w(y)\right|
\le \sum_{j=1}^k\left| 2^{j-1}w(\tau_0/2^{j-1}) - 2^j w(\tau_0/2^j) \right|\\
\le \tau_0^{\alpha+\beta} K\sum_{j=1}^k 2^{j\left(1-(\alpha+\beta)\right)-1}
\le N K 2^{k\left(1-(\alpha+\beta)\right)} \tau_0^{\alpha+\beta},
\end{multline*}
where $N=N(\alpha+\beta)$.
Note that $|w(\tau_0)|\le 2 K_0$. Then using the above inequalities and \eqref{eq0302_1}, we obtain
\begin{multline*}
|w(y)| \le 2^{-k}\left|2^k w(y) - w(\tau_0)\right|
+ 2^{-k}|w(\tau_0)|\\
\le N K 2^{-k(\alpha+\beta)} \tau_0^{\alpha+\beta} + 2K y^{\alpha+\beta}
= N K y^{\alpha+\beta},
\end{multline*}
where $N=N(\alpha+\beta)$. The inequality \eqref{eq0301} is proved.

To prove the case $\alpha+\beta=1$, we again consider $|f(x+y)-f(x)|$, where $y > 0$. If $y > 1/2$,
$$
\left|f(x+y)-f(x)\right| \le y^{\gamma}[f]_{C^{\gamma}}
\le  2^{1-\gamma}  y[f]_{C^{\gamma}}.
$$
Thus the inequality \eqref{eq0302_2} follows. For $0 < y < 1/2$, we find a positive integer $k$ satisfying
\begin{equation}
								\label{eq0302_3}
2^{-k-1} \le y < 2^{-k}.	
\end{equation}
By using the definitions of $\tau_0$ and $w$, and the same steps as above, we obtain
\begin{equation}
								\label{eq0302_8}
|w(y)| \le K k 2^{-k-1} \tau_0 + 2^{-k}|w(\tau_0)|.	
\end{equation}
Observe that from \eqref{eq0302_3}
$$
1/2 \le \tau_0<1,\quad k 2^{-k} \tau_0 = k y \le \frac{|\log y|}{\log 2} y,
$$
$$
2^{-k}|w(\tau_0)|
= 2^{-k}|f(x+\tau_0)-f(x)| \le 2^{-k}\tau_0^{\gamma}[f]_{C^{\gamma}}
\le 2^{1-\gamma}y [f]_{C^{\gamma}}.
$$
Upon combining these inequalities with \eqref{eq0302_8}, we finally obtain the inequality \eqref{eq0302_2}.
\end{proof}

\mysection{Proofs of Theorems \ref{mainth00} and \ref{mainth01}}
                                \label{sec3}
This section is devoted to the proofs of Theorems \ref{mainth00} and \ref{mainth01}.
Throughout the section, as in Theorems \ref{mainth00} and \ref{mainth01}, the kernel $K=K(y)$ in \eqref{eq23.22.24} satisfies the ellipticity condition \eqref{eq1211} and, if $\sigma=1$, the cancellation condition \eqref{eq21.47},
except in the following maximum principle for the operator $L$, where $K$ only satisfies the upper bound in \eqref{eq1211} and depends on $x$ as well.
Lemma \ref{lem2.3} may be found in the literature, but for completeness we give a short proof.

\begin{lemma}[Maximum principle]
                                    \label{lem2.3}
Let $0< \sigma < 2$ and $\lambda>0$.
Assume that
$$
0\le K(x,y)\le (2-\sigma)\Lambda |y|^{-d-\sigma}
$$
and $u\in C^2_{\text{loc}}\cap L_\infty$ satisfies
\begin{equation*}
Lu-\lambda u\ge 0\quad\text{in}\,\,\bR^d.
\end{equation*}
Then we have $u\le 0$ in $\bR^d$.
\end{lemma}

\begin{proof}
Assume that $\sup_{\bR^d} u = a > 0$.
Then there is a point $x_0 \in \bR^d$ such that $u(x_0) > a-\varepsilon$.
Let $\eta$ be an infinitely differentiable function defined on $\bR^d$ with compact support and $\eta(0) = 1$. Then the function
$$
w(x):= u(x) + 2 \varepsilon \eta(x - x_0) - a
$$
achieves a positive maximum at some point $x_1 \in \bR^d$.
Now we notice that
$$
(L-\lambda)w(x_1) \ge 2 \varepsilon L \eta(x_1 - x_0) + \lambda a > 0
$$
if $\varepsilon$ is sufficiently small.
This, however, contradicts the obvious fact that
$(L-\lambda)w(x_1) \le 0$.
Therefore, $\sup_{\bR^d} u \le 0$.\footnote{We thank an anonymous referee who pointed out to us this short proof.}
\end{proof}

Define $\omega(x) = 1/(1+|x|^{d+\sigma})$.
The following a priori $C^\alpha$ estimate is obtained in \cite[Corollary 4.3 and Remark 2.3]{DK11_01}. A part of the proof adapts an idea in \cite{Ba11}.

\begin{lemma}[H\"older estimate]
								\label{lem3}
Let $\sigma \in (0, 2)$, $\lambda\ge 0$, $f \in L_{\infty}(B_1)$, and
$u \in C^{2}_{\text{loc}}(B_1)\cap L_1(\bR^d,\omega)$
such that
$$
Lu-\lambda u  = f
$$		
in $B_1$.
Then for any $\alpha_0 \in (0, \min\{1,\sigma\})$,
we have
\begin{equation*}
[u]_{C^{\alpha_0}(B_{1/2})} \le N \|u\|_{L_1(\bR^d, \omega)} + N\osc_{B_1}f,
\end{equation*}
where $N=N(d,\nu,\Lambda,\sigma, \alpha_0)$ is uniformly bounded as $\sigma\to 2$.
\end{lemma}

We derive the following proposition from Lemma \ref{lem3}.

\begin{proposition}
                                            \label{cor1}
Let $\sigma \in (0, 2)$, $\lambda\ge 0$, $f \in L_{\infty}(B_1)$, and $u \in C^{2}_{\text{loc}}(B_1)\cap L_1(\bR^d,\omega)$ such that
$$
Lu-\lambda u  = f
$$		
in $B_1$.
Then for any $\alpha_0 \in (0, \min\{1,\sigma\})$,
we have
\begin{equation*}
[u]_{C^{\alpha_0}(B_{1/2})} \le N \sum_{k=1}^\infty 2^{-k\sigma}\dashint_{B_{2^k}}\big|u-(u)_{B_{2^k}}\big|\,dx + N\osc_{B_1}f,
\end{equation*}
where $N=N(d,\nu,\Lambda,\sigma, \alpha_0)$ is uniformly bounded as $\sigma\to 2$.
\end{proposition}
Observe that as $u\in L_1(\bR^d,\omega)$, the summation above is convergent.
Indeed, for $k \ge 2$,
\begin{align*}
&\dashint_{B_{2^k}}\big|u-(u)_{B_{2^k}}\big|\,dx
\le N 2^{-kd}\int_{B_{2^k}}|u|\,dx\\
&= N 2^{-kd} \sum_{j=1}^{k-1} \int_{B_{2^{j+1}}\setminus B_{2^j}}|u|\,dx
+ N 2^{-kd} \int_{B_2} |u| \, dx\\
&\le N \sum_{j=1}^{k-1} 2^{-kd+j(d+\sigma)}\int_{B_{2^{j+1}}\setminus B_{2^j}}\frac{|u|}{1+|x|^{d+\sigma}}\,dx
+ N 2^{-kd} \int_{B_2} \frac{|u|}{1+|x|^{d+\sigma}}\, dx.
\end{align*}
Then
\begin{multline*}
\sum_{k=1}^\infty 2^{-k\sigma}\dashint_{B_{2^k}}\big|u-(u)_{B_{2^k}}\big|\,dx
\le N \sum_{j=1}^\infty\int_{B_{2^{j+1}}\setminus B_{2^j}}\frac{|u|}{1+|x|^{d+\sigma}}\,dx\\
+ N \int_{B_2} \frac{|u|}{1+|x|^{d+\sigma}}\,dx
\le N \|u\|_{L_1(\bR^d, \omega)},
\end{multline*}
where $N=N(d,\sigma)$.

\begin{proof}[Proof of Proposition \ref{cor1}]
We define a function $v$ in $\bR^d$ as follows:
\begin{align*}
v&=u-(u)_{B_2}\quad \text{in}\quad B_2,\\
v&=u-(u)_{B_{2^k}}\quad \text{in}\quad B_{2^k}\setminus B_{2^{k-1}},\,\,k=2,3,\ldots.
\end{align*}
From $u\in C_{\text{loc}}^2(B_1)\cap L_1(\bR^d,\omega)$, it is easily seen that
$v \in C^{2}_{\text{loc}}(B_1)\cap L_1(\bR^d,\omega)$.
Now a simple computation shows that $v$ satisfies
$$
Lv-\lambda v=\tilde f:=f+\lambda (u)_{B_2}-g\quad \text{in}\quad B_1,
$$
where in $B_1$
$$
g=\int_{\bR^d}\big(u(x+y) -v(x+y) -(u)_{B_2}\big)K(y)\,dy.
$$
By the definition of $v$ and \eqref{eq1211}, we have
\begin{align}
                                        \label{eq10.54}
\sup_{B_1}|g|&\le N\sum_{k=2}^\infty 2^{-k\sigma}\big|(u)_{B_{2^k}}-(u)_{B_2}\big|\nonumber\\
&= N\sum_{k=2}^\infty 2^{-k\sigma}\Big|\sum_{j=1}^{k-1}\left((u)_{B_{2^{j+1}}}-(u)_{B_{2^j}}\right)\Big|\nonumber\\
&\le N\sum_{j=2}^\infty \sum_{k=j}^\infty 2^{-k\sigma}\dashint_{B_{2^j}}\big|u-(u)_{B_{2^j}}\big|\,dx\nonumber\\
&\le N\sum_{k=2}^\infty 2^{-k\sigma}\dashint_{B_{2^k}}\big|u-(u)_{B_{2^k}}\big|\,dx.
\end{align}
Similarly, we estimate $\|v\|_{L_1(\bR^d, \omega)}$ by
\begin{align}
                                    \label{eq11.07}
\|v\|_{L_1(\bR^d, \omega)}&=\int_{B_2}|u-(u)_{B_2}|\omega\,dx
+\sum_{k=2}^\infty\int_{B_{2^k}\setminus B_{2^{k-1}}}\big|u-(u)_{B_{2^k}}\big|\omega\,dx\nonumber\\
&\le N\sum_{k=1}^\infty 2^{-k\sigma}\dashint_{B_{2^k}}\big|u-(u)_{B_{2^k}}\big|\,dx.
\end{align}
Therefore, by Lemma \ref{lem3} applied to $v$ and using \eqref{eq10.54}, \eqref{eq11.07}, we get
\begin{multline*}
[u]_{C^{\alpha_0}(B_{1/2})}=[v]_{C^{\alpha_0}(B_{1/2})} \le N \|v\|_{L_1(\bR^d, \omega)} + N\osc_{B_1}\tilde f\\
\le  N\sum_{k=1}^\infty 2^{-k\sigma}\dashint_{B_{2^k}}\big|u-(u)_{B_{2^k}}\big|\,dx + N\osc_{B_1}f.
\end{multline*}
The proposition is proved.
\end{proof}


We are now ready to complete the proof of Theorem \ref{mainth00}.

\begin{proof}[Proof of Theorem \ref{mainth00}]
By mollifications, we may assume $u,f\in C^\infty$  with bounded derivatives. Furthermore, the case $\lambda=0$ follows from the case $\lambda>0$ by taking the limit as $\lambda\searrow 0$. So in the sequel we assume $\lambda>0$. Now applying the maximum principle Lemma \ref{lem2.3} to $u-\lambda^{-1}\|f\|_{L_\infty}$, we get $u\le \lambda^{-1}\|f\|_{L_\infty}$ in $\bR^d$. Similarly, applying Lemma \ref{lem2.3} to $-u-\lambda^{-1}\|f\|_{L_\infty}$ yields $-u\le \lambda^{-1}\|f\|_{L_\infty}$ in $\bR^d$. Thus we obtain the $L_\infty$ estimate \eqref{eq07.09.53}.

We now prove the second assertion. Take $\alpha_0=\min\{1,\sigma\}/2>0$. Let $w$ be the unique $H_2^\sigma$ strong solution to the equation
$$
Lw-\lambda w=\big(f-(f)_{B_2}\big)\eta\quad \text{in}\quad\bR^d,
$$
where $\eta$ is an infinitely differentiable function satisfying
$$
0 \le \eta \le 1,
\quad
\eta = 1\quad\text{on}\,\,B_2,
\quad
\eta = 0\quad\text{outside}\,\,B_3.
$$
Since $\left(f - (f)_{B_2}\right)\eta \in C_0^{\infty}$, by the classical theory we have $w \in H_2^\sigma \cap C_{\text{loc}}^{\infty}$.
By the $L_2$ estimate (cf. \cite[Proposition 3.5]{DK11_01}), we have
\begin{equation*}
\|(-\Delta)^{\sigma/2}w\|_{L_2}+\lambda \|w\|_{L_2}\le N(d,\nu)\|\eta\big(f-(f)_{B_2}\big)\|_{L_2}\le N(d,\nu)\omega_f(2),
\end{equation*}
which implies for any $R>0$,
\begin{equation}
            \label{eq11.55z}
\dashint_{B_R}\left(\big|(-\Delta)^{\sigma/2}w\big|^2+\lambda^2 |w|^2\right)\,dx \le N(d,\nu)R^{-d}\big(\omega_f(2)\big)^2.
\end{equation}
Set $v:=u-w$. Then $v \in C_{\text{loc}}^\infty$ satisfies
\begin{equation}
                                            \label{eq12.33}
Lv-\lambda v=\tilde f:=f(1-\eta)+(f)_{B_2}\eta.
\end{equation}
For convenience of notation, we define
$$
u'=(-\Delta)^{\sigma/2}u,\quad v'=(-\Delta)^{\sigma/2}v,\quad
w'=(-\Delta)^{\sigma/2}w.
$$
Thanks to H\"older's inequality, we have $w,w'\in L_1(\bR^d,\omega)$.
Then using the fact that $u$ and its derivatives are bounded, we see that $v,v'\in L_1(\bR^d,\omega)$.

We take the fractional derivative $(-\Delta)^{\sigma/2}$ on both sides of \eqref{eq12.33} and obtain
$$
(L-\lambda)v'=(-\Delta)^{\sigma/2}\tilde f.
$$
Applying Proposition \ref{cor1} to the equation above gives
\begin{equation}
                                    \label{eq13.31}
[v']_{C^{\alpha_0}(B_{1/2})}
\le N \sum_{k=1}^\infty 2^{-k\sigma}\dashint_{B_{2^k}}\big|v'-(v')_{B_{2^k}}\big|\,dx + N\osc_{B_1}(-\Delta)^{\sigma/2}\tilde f.
\end{equation}
We estimate the second term on the right-hand side of \eqref{eq13.31} as follows:
\begin{align*}
\osc_{B_1}(-\Delta)^{\sigma/2}\tilde f&\le 2\sup_{B_1}\big|(-\Delta)^{\sigma/2}\tilde f\big|\\
&\le N\sup_{x\in B_1}\int_{\bR^d}|y|^{-d-\sigma}\big|\tilde f(x+y)-\tilde f(x)\big|\,dy.
\end{align*}
Since for any $x,y\in B_1$, $\tilde f(x+y)=\tilde f(x)=(f)_{B_2}$, we get
\begin{equation}
								\label{eq0302_7}
\osc_{B_1}(-\Delta)^{\sigma/2}\tilde f
\le N\sup_{x\in B_1}\int_{B_1^c}|y|^{-d-\sigma}\big|\tilde f(x+y)-(f)_{B_2}\big|\,dy.		 \end{equation}
To calculate the right-hand side of the above inequality, we observe that, for any $x \in B_1$,
\begin{align}
                                        \label{eq13.48z}
&\int_{B_1^c}|y|^{-d-\sigma}\big|\tilde f(x+y)-(f)_{B_2}\big|\,dy\nonumber\\
&= \int_{B_1^c}|y|^{-d-\sigma}\big|1-\eta(x+y)\big| \big|f(x+y)-(f)_{B_2}\big|\,dy\nonumber\\
&\le \int_{B_1^c}|y|^{-d-\sigma}\big|f(x+y)-(f)_{B_2}\big|\,dy\le N\sum_{k=1}^\infty 2^{-k\sigma}\omega_f(2^k).
\end{align}
Combining \eqref{eq13.31}, \eqref{eq0302_7}, and \eqref{eq13.48z} yields
\begin{equation}
                                    \label{eq13.31zz}
[v']_{C^{\alpha_0}(B_{1/2})}
\le N\tilde \omega_f(2)
+N \sum_{k=1}^\infty 2^{-k\sigma}\dashint_{B_{2^k}}\big|v'-(v')_{B_{2^k}}\big|\,dx,
\end{equation}
where
$$
\tilde \omega_f(\cdot):=\sum_{k=0}^\infty 2^{-k\sigma}\omega_f(2^k\cdot)
$$
is also a Dini function thanks to Lemma \ref{lem4.08}. Let $s\ge 1$ be an integer to be specified later. By the triangle inequality, \eqref{eq11.55z}, and \eqref{eq13.31zz},
\begin{align*}
&\dashint_{B_{2^{-s}}}\big|u'-(u')_{B_{2^{-s}}}\big|\,dx\\
&\le \dashint_{B_{2^{-s}}}\big|v'-(v')_{B_{2^{-s}}}\big|\,dx
+N2^{sd/2}\omega_f(2)\\
&\le N2^{-s\alpha_0}[v']_{C^{\alpha_0}(B_{1/2})}
+N2^{sd/2}\omega_f(2)\\
&\le N2^{-s\alpha_0}\sum_{k=1}^\infty 2^{-k\sigma}\dashint_{B_{2^k}}\big|v'-(v')_{B_{2^k}}\big|\,dx +N2^{sd/2}\tilde \omega_f(2).
\end{align*}
By using the triangle inequality and \eqref{eq11.55z} again, we obtain
\begin{align}
&\dashint_{B_{2^{-s}}}\big|u'-(u')_{B_{2^{-s}}}\big|\,dx\nonumber\\
                                            \label{eq13.07}
&\le N2^{-s\alpha_0}\sum_{k=1}^\infty 2^{-k\sigma}\dashint_{B_{2^k}}\big|u'-(u')_{B_{2^k}}\big|\,dx +N2^{sd/2}\tilde \omega_f(2).
\end{align}
For any integer $l$, denote
$$
M_l=\sup_{x_0\in \bR^d}\dashint_{B_{2^{l}}(x_0)}\big|u'-(u')_{B_{2^{l}}(x_0)}\big|\,dx,
$$
$$
\tilde M_l=\sum_{j=0}^\infty 2^{-j\sigma/2}M_{l+j},\quad
\tilde{\tilde \omega}_f(\cdot)=\sum_{j=0}^\infty 2^{-j\sigma/2}\tilde \omega_f(2^j\cdot).
$$
By Lemma \ref{lem4.08}, $\tilde{\tilde \omega}_f$ is a Dini function.
From \eqref{eq13.07}, a scaling with a shift of the coordinates gives, for any integer $l$,
\begin{equation*}
M_l\le N2^{sd/2}\tilde \omega_f(2^{1+s+l})
+N2^{-s\alpha_0}\sum_{k=1}^\infty 2^{-k\sigma}M_{k+s+l}.
\end{equation*}
Consequently, by a change of indices $j+k\to j,k\to k$,
\begin{align}
                        \label{eq13.13}
\tilde M_l&\le N2^{sd/2}\tilde{\tilde \omega}_f(2^{1+s+l})
+N2^{-s\alpha_0}\sum_{j=0}^\infty\sum_{k=1}^\infty 2^{-j\sigma/2-k\sigma}M_{k+s+l+j}\nonumber\\
&= N2^{sd/2}\tilde{\tilde \omega}_f(2^{1+s+l})
+N2^{-s\alpha_0}\sum_{j=1}^\infty 2^{-j\sigma/2}M_{s+l+j}\sum_{k=1}^j 2^{-k\sigma/2}\nonumber\\
&\le N2^{sd/2}\tilde{\tilde \omega}_f(2^{1+s+l})
+N2^{-s\alpha_0}\tilde M_{s+l},
\end{align}
where $N=N(d,\nu,\Lambda,\sigma)$.
We now fix $s$ sufficiently large such that $N2^{-s\alpha_0}\le 1/2$. Since $\tilde M_l,l\in \bZ$ is a bounded sequence, it then follows from \eqref{eq13.13} that for any integer $l$,
$$
\tilde M_l\le N\tilde{\tilde \omega}_f(2^{1+s+l})+\frac 1 2 \tilde M_{s+l}\le\ldots
\le N\sum_{j=1}^\infty 2^{-j}\tilde{\tilde \omega}_f(2^{1+js+l}),
$$
which together with the obvious inequality $M_l\le \tilde M_l$ implies
$$
M_l\le N\sum_{j=1}^\infty 2^{-j}\tilde{\tilde \omega}_f(2^{1+js+l}).
$$
Since $\sum_{j=1}^\infty 2^{-j}\tilde{\tilde \omega}_f(2^{1+js}\cdot)$ is also a Dini function by Lemma \ref{lem4.08}, it follows from Lemma \ref{lem4.09} that $u'$ is uniformly continuous in $\bR^d$. Thus we have obtained an a priori estimate of $\omega_{u'}$.

It remains to get an a priori estimate of $\|u'\|_{L_\infty}$. We take a nonnegative cutoff function $\eta_1\in C_0^\infty(B_2)$ with a unit integral. By the triangle inequality, for any $x_0\in \bR^d$,
\begin{multline*}
|u'(x_0)|\le |u'(x_0)-\eta_1*u'(x_0)|+|\eta_1*u'(x_0)|\\
\le \omega_{u'}(2)+|(-\Delta)^{\sigma/2} \eta_1 *u(x_0)|
\le \omega_{u'}(2)+N\|u\|_{L_\infty}.
\end{multline*}
By keeping track of the constants, one can see that the constants $N$ above are uniformly bounded as $\sigma\to 2$.
The theorem is proved.
\end{proof}


We finish this section by giving the proof of Theorem \ref{mainth01}.

\begin{proof}[Proof of Theorem \ref{mainth01}]
As in the proof of Theorem \ref{mainth00}, we may assume $u,f\in C^\infty$  with bounded derivatives and $\lambda>0$.

First we consider the case when $\alpha\in (0,\min\{1,\sigma\})$. Take $\alpha_0=(\alpha+\min\{1,\sigma\})/2>\alpha$. Our aim is to estimate
$$
\int_{B_r}
\big|u-(u)_{B_r}\big| \,dx,\quad
\int_{B_r}\big|(-\Delta)^{\sigma/2}u-\big((-\Delta)^{\sigma/2}u\big)_{B_r}\big|\,dx
$$
for $r>0$.

We define functions $w$ and $v$ in the same way as in the proof of Theorem \ref{mainth00}. Instead of \eqref{eq11.55z}, since $\|\eta \left( f - (f)_{B_2} \right) \|_{L_2} \le N [f]_{C^\alpha}$, now we have
\begin{equation}
            \label{eq11.55}
\dashint_{B_R}\left(\big|(-\Delta)^{\sigma/2}w\big|^2+\lambda^2 |w|^2\right)\,dx \le NR^{-d}[f]^2_{C^\alpha}.
\end{equation}
Since $\osc_{B_1}\tilde f=0$, it follows from Proposition \ref{cor1} that
\begin{equation}
                                        \label{eq13.30z}
[v]_{C^{\alpha_0}(B_{1/2})} \le N \sum_{k=1}^\infty 2^{-k\sigma}\dashint_{B_{2^k}}\big|v-(v)_{B_{2^k}}\big|\,dx.
\end{equation}
Let $r_0\in (0,1/2)$ be a constant to be specified later.
By the triangle inequality, \eqref{eq11.55}, and \eqref{eq13.30z}, we obtain
\begin{align*}
&\dashint_{B_{r_0}}\big|u-(u)_{B_{r_0}}\big| \,dx\\
&\le \dashint_{B_{r_0}}\big|v-(v)_{B_{r_0}}\big| \,dx+N\lambda^{-1}r_0^{-d/2}[f]_{C^\alpha}\\
&\le Nr_0^{\alpha_0}[v]_{C^{\alpha_0}(B_{1/2})}+N\lambda^{-1}r_0^{-d/2}[f]_{C^\alpha}\\
&\le Nr_0^{\alpha_0}\sum_{k=1}^\infty 2^{-k\sigma}\dashint_{B_{2^k}}\big|v-(v)_{B_{2^k}}\big|\,dx
+N\lambda^{-1}r_0^{-d/2}[f]_{C^\alpha}.
\end{align*}
By using the triangle inequality and \eqref{eq11.55} again, we get
\begin{align*}
&\dashint_{B_{r_0}}\big|u-(u)_{B_{r_0}}\big| \,dx\\
&\le Nr_0^{\alpha_0}\sum_{k=1}^\infty 2^{-k\sigma}\dashint_{B_{2^k}}\big|u-(u)_{B_{2^k}}\big|\,dx
+N\lambda^{-1}\big(r_0^{\alpha_0}+r_0^{-d/2}\big)[f]_{C^\alpha}\\
&\le Nr_0^{\alpha_0}\sum_{k=1}^\infty 2^{-k\sigma+k\alpha}[u]_{C^\alpha}
+N\lambda^{-1}\big(r_0^{\alpha_0}+r_0^{-d/2}\big)[f]_{C^\alpha}\\
&\le Nr_0^{\alpha_0}[u]_{C^\alpha}
+N\lambda^{-1}r_0^{-d/2}[f]_{C^\alpha},
\end{align*}
where in the last inequality we have used $\alpha<\sigma$.
For any $r>0$, by scaling $x\to xr/r_0$, we have
\begin{align*}
&r^{-\alpha}\dashint_{B_{r}}\big|u-(u)_{B_{r}}\big| \,dx\le N r_0^{\alpha_0-\alpha}[u]_{C^\alpha}
+N\lambda^{-1}r_0^{-\alpha-d/2}[f]_{C^\alpha}.
\end{align*}
Shifting the coordinates, we get for any $x_0\in \bR^d$ and $r>0$.
\begin{equation*}
r^{-\alpha}\dashint_{B_{r}(x_0)}\big|u-(u)_{B_{r}(x_0)}\big| \,dx\le N r_0^{\alpha_0-\alpha}[u]_{C^\alpha}
+N\lambda^{-1}r_0^{-\alpha-d/2}[f]_{C^\alpha}.
\end{equation*}
We take the supremum of the left-hand side above
with respect to $x_0\in \bR^d$ and $r > 0$ and then use Campanato's equivalent definition of H\"older norms to get
\begin{align*}
[u]_{C^\alpha}
\le N r_0^{\alpha_0-\alpha}[u]_{C^\alpha}
+N\lambda^{-1}r_0^{-\alpha-d/2}[f]_{C^\alpha}.
\end{align*}
Recall that $\alpha<\alpha_0$.
Upon taking $r_0$ sufficiently small such that $Nr_0^{\alpha_0-\alpha}\le 1/2$, we obtain
\begin{equation}
                                \label{eq14.48}
\lambda [u]_{C^\alpha}\le N[f]_{C^\alpha}.
\end{equation}

Next we estimate the H\"older norm of the fractional derivative $(-\Delta)^{\sigma/2}u$.
Instead of \eqref{eq13.48z}, for any $x \in B_1$ we have
\begin{align}
                                        \label{eq13.48}
&\int_{B_1^c}|y|^{-d-\sigma}\big|\tilde f(x+y)-(f)_{B_2}\big|\,dy\nonumber\\
&= \int_{B_1^c}|y|^{-d-\sigma}\big|1-\eta(x+y)\big| \big|f(x+y)-(f)_{B_2}\big|\,dy\nonumber\\
&\le N[f]_{C^\alpha}\int_{B_1^c}|y|^{-d-\sigma+\alpha}\,dy\le N[f]_{C^\alpha}.
\end{align}
In the last inequality above, we used the assumption that $\alpha<\sigma$.
Combining \eqref{eq13.31}, \eqref{eq0302_7}, and \eqref{eq13.48} yields
\begin{equation}
                                    \label{eq13.31z}
[v']_{C^{\alpha_0}(B_{1/2})}
\le N[f]_{C^\alpha}
+N \sum_{k=1}^\infty 2^{-k\sigma}\dashint_{B_{2^k}}\big|v'-(v')_{B_{2^k}}\big|\,dx.
\end{equation}
Now we argue as in the estimate of $[u]_{C^\alpha}$. By the triangle inequality, \eqref{eq11.55}, and \eqref{eq13.31z},
\begin{align*}
&\dashint_{B_{r_0}}\big|u'-(u')_{B_{r_0}}\big|\,dx\\
&\le \dashint_{B_{r_0}}\big|v'-(v')_{B_{r_0}}\big|\,dx
+Nr_0^{-d/2}[f]_{C^\alpha}\\
&\le Nr_0^{\alpha_0}[v']_{C^{\alpha_0}(B_{1/2})}
+Nr_0^{-d/2}[f]_{C^\alpha}\\
&\le Nr_0^{\alpha_0}\sum_{k=1}^\infty 2^{-k\sigma}\dashint_{B_{2^k}}\big|v'-(v')_{B_{2^k}}\big|\,dx +N(r_0^{\alpha_0}+r_0^{-d/2})[f]_{C^\alpha}.
\end{align*}
By using the triangle inequality and \eqref{eq11.55} again, we get
\begin{align*}
&\dashint_{B_{r_0}}\big|u'-(u')_{B_{r_0}}\big|\,dx\\
&\le Nr_0^{\alpha_0}\sum_{k=1}^\infty 2^{-k\sigma}\dashint_{B_{2^k}}\big|u'-(u')_{B_{2^k}}\big|\,dx
+N(r_0^{\alpha_0}+r_0^{-d/2})[f]_{C^\alpha}\\
&\le Nr_0^{\alpha_0}\sum_{k=1}^\infty 2^{-k\sigma+k\alpha}[u']_{C^\alpha}
+N\big(r_0^{\alpha_0}+r_0^{-d/2}\big)[f]_{C^\alpha}\\
&\le Nr_0^{\alpha_0}[u']_{C^\alpha}
+Nr_0^{-d/2}[f]_{C^\alpha}.
\end{align*}
For any $r>0$, by scaling $x\to xr/r_0$,
\begin{align*}
&r^{-\alpha}\dashint_{B_{r}}\big|u'-(u')_{B_{r}}\big| \,dx\le N r_0^{\alpha_0-\alpha}[u']_{C^\alpha}
+Nr_0^{-\alpha-d/2}[f]_{C^\alpha}.
\end{align*}
Shifting the coordinates, we get for any $x_0\in \bR^d$ and $r>0$.
\begin{align*}
&r^{-\alpha}\dashint_{B_{r}(x_0)}\big|u'-(u')_{B_{r}(x_0)}\big| \,dx\le N r_0^{\alpha_0-\alpha}[u']_{C^\alpha}
+Nr_0^{-\alpha-d/2}[f]_{C^\alpha}.
\end{align*}
We take the supremum of the left-hand side above
with respect to $x_0\in \bR^d$ and $r > 0$ and then use Campanato's equivalent definition of H\"older norms to get
\begin{align*}
[u']_{C^\alpha}
\le N r_0^{\alpha_0-\alpha}[u']_{C^\alpha}
+Nr_0^{-\alpha-d/2}[f]_{C^\alpha}.
\end{align*}
Upon taking $r_0$ sufficiently small such that $Nr_0^{\alpha_0-\alpha}\le 1/2$, we obtain
\begin{equation}
                                \label{eq14.48z}
[u']_{C^\alpha}\le N[f]_{C^\alpha}.
\end{equation}
Collecting \eqref{eq14.48} and \eqref{eq14.48z}, we reach \eqref{eq24.09.53} in the case $\alpha<\min\{1,\sigma\}$.

For general $\alpha\in (0,1]$, we fix $\beta \in (0,\alpha)$ such that $\alpha-\beta < \min\{1,\sigma\}$.
For any $h>0$ and unit vector $\xi \in \bR^d$, we have
$$
L u_{\beta,h} - \lambda u_{\beta,h} = f_{\beta,h},
$$
where
\begin{align*}
f_{\beta,h}(x) &= h^{-\beta}\left(f(x+h\xi)-f(x)\right),\\
u_{\beta,h}(x) &= h^{-\beta}\left(u(x+h\xi)-u(x)\right).
\end{align*}
By Lemma \ref{lem1}, $f_{\beta,h} \in C^{\alpha-\beta}$
and $[f_{\beta,h}]_{C^{\alpha-\beta}} \le 2[f]_{C^{\alpha}}$.
Then the Schauder estimate \eqref{eq24.09.53} proved above for the case $\alpha-\beta<\min\{1,\sigma\}$ gives
\begin{equation}
                                        \label{eq11.34}
[(-\Delta)^{\sigma/2}u_{\beta,h}]_{C^{\alpha-\beta}}+\lambda [u_{\beta,h}]_{C^{\alpha-\beta}}
\le N [f_{\beta,h}]_{C^{\alpha-\beta}}\le N[f]_{C^{\alpha}}.
\end{equation}
Since $u$ belongs to $C^\infty$ with bounded derivatives, $(-\Delta)^{\sigma/2}u$ is a bounded function in $\bR^d$.
Now let $y\in \bR^d$ be a vector in the $\xi$-direction.
For $\alpha\in (0,1)$, thanks to Lemma \ref{lem2} i) and \eqref{eq11.34}, we have
$$
|(-\Delta)^{\sigma/2}u(x+y)-(-\Delta)^{\sigma/2}u(x)|+\lambda |u(x+y)-u(x)|
\le  N|y|^\alpha [f]_{C^{\alpha}}.
$$
Similarly, for $\alpha = 1$, by Lemma \ref{lem2} ii) and \eqref{eq11.34} we have
\begin{multline*}
|(-\Delta)^{\sigma/2}u(x+y)-(-\Delta)^{\sigma/2}u(y)|+\lambda |u(x+y)-u(x)|\\
\le N\big|y \log |y|\big| [f]_{C^{0,1}}
+ N |y| \big([(-\Delta)^{\sigma/2}u]_{C^{1/2}}+\lambda [u]_{C^{1/2}}\big),
\end{multline*}
where, by \eqref{eq24.09.53} with $\alpha = 1/2$
$$
[(-\Delta)^{\sigma/2}u]_{C^{1/2}}+\lambda [u]_{C^{1/2}} \le N [f]_{C^{1/2}} \le N\|f\|_{C^{0,1}}.
$$
Note that for $|y| \le 1/2$, $|y| \le N \big|y \log|y| \big|$. Therefore, for $|y| \le 1/2$,
$$
|(-\Delta)^{\sigma/2}u(x+y)-(-\Delta)^{\sigma/2}u(y)|+\lambda |u(x+y)-u(x)|
\le N\big|y \log |y|\big|\|f\|_{C^{0,1}}.
$$
Because $\xi$ is an arbitrary unit vector in $\bR^d$, we finally conclude \eqref{eq24.09.53} and \eqref{eq24.09.53z}. As before, by keeping track of the constants, one can see that the constants $N$ above are uniformly bounded as $\sigma\to 2$. The theorem is proved.
\end{proof}

\begin{remark}
                                    \label{rm3.33}
We give a proof of the claim in Remark \ref{rm0}. In fact, we prove a more general result under a relaxed ellipticity condition. Instead of assuming \eqref{eq1211} for all $y\in \bR^d$, we only assume the lower bound
$$
(2-\sigma)\frac{\nu}{|y|^{d+\sigma}}
\le K(y)
$$
for any $y\in B_1$ and the upper bound
$$
K(y) \le (2-\sigma)\frac{\Lambda}{|y|^{d+\sigma}}
$$ for all $y\in \bR^d$. In particular, we allow $K$ to vanish outside $B_1$. We fix a constant $\lambda>0$.
The $L_\infty$ estimate \eqref{eq07.09.53} still holds in this case as the proof of the maximum principle Lemma \ref{lem2.3} does not use the lower bound of $K$.
Now let $L_1$ be the elliptic operator with kernel $K_1(y):=K(y)+I_{B_1^c}(2-\sigma)|y|^{-d-\sigma}$. Then $u$ satisfies
$$
L_1 u-\lambda u=\tilde f\quad \text{in}\,\,\bR^d,
$$
where $\tilde f=f+L_2 u$ and $L_2$ is the operator with kernel
$K_2(y)=I_{B_1^c}(2-\sigma)|y|^{-d-\sigma}$. Clearly, \eqref{eq1211} holds with $K_1$ in place of $K$ for appropriate constants $\nu$ and $\Lambda$. By Theorem \ref{mainth01} i), for any $\alpha\in (0,1)$ we have
\begin{equation}
                                \label{eq24.4.40}
[(-\Delta)^{\sigma/2}u]_{C^\alpha}+\lambda [u]_{C^\alpha}
\le N [\tilde f]_{C^\alpha}
\le N [f]_{C^\alpha}+N[L_2 u]_{C^\alpha}.
\end{equation}
From the choice of $K_2$, it is easily seen that when $\sigma\in (0,1]$,
$$
[L_2 u]_{C^\alpha}\le N[u]_{C^\alpha}.
$$
Therefore, the last term on the right-hand side of \eqref{eq24.4.40} can be absorbed to the left-hand side if $\lambda$ is sufficiently large. For general $\lambda>0$, by the interpolation inequality (see, for instance, Lemma 4 in \cite{MP92}),
\begin{equation}
                                        \label{eq24.5.07}
[u]_{C^\alpha}\le \varepsilon [(-\Delta)^{\sigma/2}u]_{C^\alpha}+N(\varepsilon)\|u\|_{L_\infty}
\end{equation}
for any $\varepsilon\in (0,1)$. Combining \eqref{eq24.4.40}, \eqref{eq24.5.07} with a sufficiently small $\varepsilon$, and \eqref{eq07.09.53} yield
\begin{equation}
                                \label{eq24.4.40z}
[(-\Delta)^{\sigma/2}u]_{C^\alpha}+\lambda [u]_{C^\alpha}
\le N [f]_{C^\alpha}+N\lambda^{-1}\|f\|_{L_\infty}.
\end{equation}
In the case when $\sigma\in (1,2)$, we have
$$
[L_2 u]_{C^\alpha}\le N([u]_{C^\alpha}+[Du]_{C^\alpha}),
$$
and the estimate \eqref{eq24.4.40z} follows in a similar way by using the interpolation inequality
\begin{equation*}
[u]_{C^\alpha}+[Du]_{C^\alpha}\le \varepsilon [(-\Delta)^{\sigma/2}u]_{C^\alpha}+N(\varepsilon)\|u\|_{L_\infty}.
\end{equation*}
Finally, as in the proof of Theorem \ref{mainth01}, we also get \eqref{eq24.09.53z} with a constant $N$ depending also on $\lambda$. The claim is proved.
\end{remark}

\mysection{Solvability in Lipschitz--Zygmund spaces}
                            \label{sec4}

As an application of the a priori estimates in Theorem \ref{mainth01}, in this section we prove a solvability result for the non-local equation \eqref{elliptic} in $\bR^d$, i.e., Theorem \ref{thm2}. We introduce a few more notation. For any function $u\in L_\infty(\bR^d)$, denote $U(x,y)$ to be the harmonic extension of $u$ to $\bR^{d+1}_+$:
$$
U(\cdot,y)=P(\cdot,y)*u(\cdot),\quad \text{for}\,\,y>0,
$$
where $P(\cdot,y)$ is the Poisson kernel on $\bR^{d+1}_+$.
For $\alpha>0$, let $k$ be the smallest integer greater than $\alpha$. We define the Lipschitz--Zygmund space by
$$
\Lambda^\alpha=\{u\in L_\infty(\bR^d)\,:\,\sup_{y>0}y^{k-\alpha}\|D^k_y U(\cdot,y)\|_{L_\infty}<\infty\},
$$
which is equipped with the norm
$$
\|u\|_{\Lambda^\alpha}:=\|u\|_{L_\infty}+\sup_{y>0}y^{k-\alpha}\|D^k_y U(\cdot,y)\|_{L_\infty}.
$$
We recall a few well known properties of the Lipschitz--Zygmund spaces; see, for instance, \cite[Chap. V]{St70} and \cite[Chap. VI]{St93}:
\begin{enumerate}
\item For any $\alpha_1>\alpha_2>0$, we have $\Lambda^{\alpha_1}\subsetneq \Lambda^{\alpha_2}$.
\item For any non-integer $\alpha>0$, $\Lambda^\alpha$ is equivalent to the H\"older space $C^\alpha$. On the other hand, for any positive integer $\alpha$, $C^{\alpha-1,1}\subsetneq \Lambda^\alpha$.
\item Let $\alpha>1$. Then $u\in \Lambda^\alpha$ if and only if $u\in L_\infty$ and $Du\in \Lambda^{\alpha-1}$. The norms $\|u\|_{\Lambda^\alpha}$ and $\|u\|_{L_\infty}+\|Du\|_{\Lambda^{\alpha-1}}$ are equivalent.
\item The norms $\|u\|_{\Lambda^\alpha}$ and
$$
\|u\|_{L_\infty}+\sup_{y>0}y^{l-\alpha}\|D^l_y U(\cdot,y)\|_{L_\infty}
$$
are equivalent for any integer $l$ greater than $\alpha$.
\item For any $\alpha\in (0,2)$, $\|u\|_{\Lambda^\alpha}$ is equivalent to the norm
$$
\|u\|_{L_\infty}+\sup_{|h|>0}|h|^{-\alpha}
\|u(\cdot+h)+u(\cdot-h)-2u(\cdot)\|_{L_\infty}.
$$
\item Let $P$ be a pseudo-differential operator of order $m$. Then $P$ is continuous from $\Lambda^{\alpha}$ to $\Lambda^{\alpha-m}$ provided that $\alpha>m$. In particular, the Bessel potential operator $(1-\Delta)^{-s/2},s\ge 0$ is an isomorphism from $\Lambda^\alpha$ to $\Lambda^{\alpha+s}$ for any $\alpha>0$.
\end{enumerate}

As before, throughout this section we assume $K(x,y)=K(y)$.
For the proof of Theorem \ref{thm2}, we need the inequality \eqref{eq22.36} below. For future reference, we present here a rather complete and generalized version of Lemmas \ref{lem1} and \ref{lem2}.

\begin{lemma}
                                    \label{lem3.2}
Let $\alpha>0$ and $\beta\in (0,1)$ be constants. Then $f\in \Lambda^{\alpha+\beta}$ if and only if $f\in L_\infty$ and $f_{\beta,h}\in \Lambda^{\alpha}$ for any $\bR^d\ni h\neq 0 $, where $f_{\beta,h}$ is defined in Lemma \ref{lem1}. If $\|f_{\beta,h}\|_{\Lambda^{\alpha}}\le K$ for any nonzero $h\in \bR^d$, then we have
\begin{equation}
                                    \label{eq22.34}
\|f\|_{\Lambda^{\alpha+\beta}}\le N(\|f\|_{L_\infty}+K).
\end{equation}
Moreover, for any $\|f\|_{\Lambda^{\alpha+\beta}}$ and nonzero $h\in \bR^d$, we have
\begin{equation}
                                            \label{eq22.36}
\|f_{\beta,h}\|_{\Lambda^{\alpha}}\le N\|f\|_{\Lambda^{\alpha+\beta}}.
\end{equation}
\end{lemma}

\begin{proof}
It suffices to show \eqref{eq22.34} and \eqref{eq22.36}. Thanks to Property (3), we may assume $\alpha\in (0,1]$, so that $\alpha+\beta<2$. The estimate \eqref{eq22.34} for $\alpha \in (0,1)$ follows immediately from Property (5) and
\begin{equation}
								\label{eq22.35}
|h|^{-(\alpha+\beta)}
\|f(\cdot+h)+f(\cdot-h)-2f(\cdot)\|_{L_\infty}=|h|^{-\alpha}
\|f_{\beta,h}(\cdot+h)-f_{\beta,h}(\cdot)\|_{L_\infty}.
\end{equation}

When $\alpha = 1$, the right-hand side of the above inequality is bounded by $4 K$ if $|h| \ge 1/2$.
For $|h| < 1/2$, take a positive integer $k$ such that $2^{-k-1} \le |h| < 2^{-k}$.
For any nonzero $y\in \bR^d$, set $T_y$ to be the shift operator $f(\cdot)\to f(\cdot+y)$.
We see that
\begin{equation}
								\label{eq0319_02}
(T_h-1)^2 = \frac 1 4 (T_{2h}-1)^2 - \frac 1 2 (T_{2h}-1)(T_h-1)^2 + \frac 1 4 (T_h-1)^4,	
\end{equation}
which follows by taking the square of the obvious identity
\begin{equation}
								\label{eq0319_01}
T_h-1=\frac 1 2\big(T_{2h}-1-(T_h-1)^2\big).								
\end{equation}
Upon setting $I_j := |2^j h|^{-\beta-1} | (T_{2^j h} - 1)^2 f|$ and using \eqref{eq0319_02}, we obtain
$$
I_j \le 2^{\beta-1}I_{j+1} + N K,
\quad
j = 1, 2, \cdots.
$$
This implies that $I_0 \le 2^{(\beta-1)(k-1)} I_k + N(\beta) K$.
Note that $2^{(\beta-1)(k-1)} \le 1$ and
$$
I_k \le \frac{2}{2^k |h|} \|f_{\beta, 2^kh}\|_{L_\infty}
\le 4 K,
$$
where the second inequality is due to the choice of $k$.
Also note that the left-hand side of \eqref{eq22.35} with $\alpha = 1$ is the sup of $I_0$ with respect to $x \in \bR^d$.
Therefore, the inequality \eqref{eq22.34} also follows when $\alpha = 1$.

It remains to prove \eqref{eq22.36}.
Due to Properties (2) and (1),
$$
\|f_{\beta,h}\|_{L_\infty}\le [f]_{C^\beta}\le N\|f\|_{\Lambda^\beta}\le N\|f\|_{\Lambda^{\alpha+\beta}}.
$$
Thus by Property (5), to prove \eqref{eq22.36}, it suffices to show that
\begin{equation}
                                \label{eq22.44}
|(T_h-1)(T_y-1)^2 f|\le N|h|^\beta |y|^\alpha\|f\|_{\Lambda^{\alpha+\beta}}
\end{equation}
for any nonzero $h,y\in \bR^d$. In the case $|h|\ge |y|$, by Property (5),
\begin{multline*}
|(T_h-1)(T_y-1)^2 f|\le 2\|(T_y-1)^2 f\|_{L_\infty}\\
\le N|y|^{\alpha+\beta}\|f\|_{\Lambda^{\alpha+\beta}}\le N|h|^\beta |y|^\alpha\|f\|_{\Lambda^{\alpha+\beta}}.
\end{multline*}
In the case $|h|<|y|$, let $k$ be a positive integer such that $2^{k-1}|h|<|y|\le 2^k |h|$. By repeatedly using the identity \eqref{eq0319_01},
we get
$$
T_h-1=-\sum_{j=0}^{k-1}{2^{-j-1}} (T_{2^jh}-1)^2+2^{-k}(T_{2^kh}-1).
$$
Therefore, by Property (5) again,
\begin{align*}
&|(T_h-1)(T_y-1)^2 f|\\
&\le \sum_{j=0}^{k-1}{2^{-j-1}} |(T_y-1)^2(T_{2^jh}-1)^2 f|+2^{-k}|(T_{2^kh}-1)(T_y-1)^2 f|\\
&\le N\Big(\sum_{j=0}^{k-1}{2^{-j-1}} 2^{j(\alpha+\beta)}|h|^{\alpha+\beta}+2^{-k}|y|^{\alpha+\beta}\Big)
\|f\|_{\Lambda^{\alpha+\beta}},
\end{align*}
which is less than the right-hand side of \eqref{eq22.44} by the choice of $k$. The lemma is proved.
\end{proof}


Denote $\cF^{-1}$ to be the operator of the inverse Fourier transform.
\begin{lemma}
                                    \label{lem3.3}
Let $G_{\beta}=\cF^{-1}\big(|\xi|^{\beta}+\lambda\big)^{-1}$ be the fundamental solution (Green's function) of the operator  $(-\Delta)^{\beta/2}+\lambda$. Then we have $G_\beta>0$ and $\|G_\beta\|_{L_1}=1/\lambda$.
\end{lemma}

\begin{proof}
From
$$
\big(|\xi|^{\beta}+\lambda\big)^{-1}=\int_0^\infty e^{-\lambda t-|\xi|^\beta t}\,dt,
$$
we get
$$
G_\beta=\int_0^\infty e^{-\lambda t}\cF^{-1} \big(e^{-|\xi|^\beta t}\big)\,dt.
$$
It is well known that
$$
\cF^{-1} \big(e^{-|\xi|^\beta t}\big)>0,\quad \big\|\cF^{-1} \big(e^{-|\xi|^\beta t}\big)\big\|_{L_1}=1.
$$
Therefore, the conclusion of the lemma follows from Fubini's theorem.
\end{proof}

We will use the following continuity estimates of $(-\Delta)^{\beta/2}$ and $L$.

\begin{lemma}
                                    \label{lem3.4}
For any $\alpha,\beta>0$, the norms $\|u\|_{\Lambda^{\alpha+\beta}}$ and
$$
\|u\|_{\Lambda^\alpha}+\|(-\Delta)^{\beta/2}u\|_{\Lambda^\alpha}
$$
are equivalent. In particular, $(-\Delta)^{\beta/2}+\lambda$ is a continuous operator from $\Lambda^{\alpha+\beta}$ to $\Lambda^{\alpha}$.
\end{lemma}
\begin{proof}
We remark that the lemma follows from the results in \cite[Chap. V.3]{St70} as well as Lemma \ref{lem3.3}; see, for instance, the proof of \cite[Lemma 12]{MP09}. Here we choose to give a self-contained argument, which also will be useful in the proof of Lemma \ref{lem3.5}.

Thanks to Property (6), we may assume $\alpha\in (0,1)$.
Take an integer $k$ greater than $\alpha+\beta$.
Then by Property (4), we have
\begin{equation}
                                        \label{eq22.00}
\|(-\Delta)^{\beta/2}u\|_{\Lambda^\alpha}\simeq\|(-\Delta)^{\beta/2}u\|_{L_\infty}
+\sup_{y>0}y^{k-\alpha}\|D^k_y (-\Delta)^{\beta/2}U(\cdot,y)\|_{L_\infty}.
\end{equation}
Due to the semi-group property of harmonic extensions,
$$
U(\cdot,y)=U(\cdot, y/2)*P(\cdot,y/2).
$$
By using the inequality
$$
\|(-\Delta)^{\beta/2}P(\cdot,y/2)\|_{L_1}\le Ny^{-\beta}
$$
and Young's inequality,
\begin{align}
                                \label{eq22.01}
&y^{k-\alpha}\|D^k_y (-\Delta_x)^{\beta/2}U(\cdot,y)\|_{L_\infty}\nonumber\\
&=y^{k-\alpha}\|D^k_y U(\cdot,y/2)*(-\Delta)^{\beta/2}P(\cdot,y/2)\|_{L_\infty}\nonumber\\
&\le y^{k-\alpha}\|D^k_y U(\cdot,y/2)\|_{L_\infty}\|(-\Delta)^{\beta/2}P(\cdot,y/2)\|_{L_1}\nonumber\\
&\le Ny^{k-\alpha-\beta}\|D^k_y U(\cdot,y/2)\|_{L_\infty}\le N\|u\|_{\Lambda^{\alpha+\beta}}.
\end{align}
We estimate the first term on the right-hand side of \eqref{eq22.00} by
\begin{equation}
									\label{eq22.02}
\|(-\Delta)^{\beta/2}U(\cdot,y)\|_{L_\infty}+Ny^\alpha\|(-\Delta)^{\beta/2}u\|_{\Lambda^\alpha},	 
\end{equation}
which is from the inequalities
\begin{multline*}
\|(-\Delta)^{\beta/2}u\|_{L_\infty} \le  \|(-\Delta)^{\beta/2}U(\cdot,y)\|_{L_\infty}
+\|(-\Delta)^{\beta/2}u(\cdot) - (-\Delta)^{\beta/2}U(\cdot,y)\|_{L_\infty}\\
\le \|(-\Delta)^{\beta/2}U(\cdot,y)\|_{L_\infty}+N \|(-\Delta)^{\beta/2}u\|_{\Lambda^\alpha} \| P(\cdot,y)|\cdot|^\alpha\|_{L_1}.
\end{multline*}
Clearly, the second term in \eqref{eq22.02} can be absorbed to the left-hand side of \eqref{eq22.00} if we choose $y$ sufficiently small, and the first term in \eqref{eq22.02} is bounded by $N\|u\|_{L_\infty}$.
This together with \eqref{eq22.01} and Property (1) gives
$$
\|u\|_{\Lambda^\alpha}+\|(-\Delta)^{\beta/2}u\|_{\Lambda^\alpha}\le N\|u\|_{\Lambda^{\alpha+\beta}}.
$$

For the other direction, it suffices to show that
for any $y>0$,
\begin{equation}
                                    \label{eq22.22}
y^{k+2l-\alpha-\beta}\|D^{k+2l}_y U(\cdot,y)\|_{L_\infty}\le N\|(-\Delta)^{\beta/2}u\|_{\Lambda^\alpha},
\end{equation}
where $k$ and $l$ are integers greater than $\alpha$ and $\beta/2$, respectively. As before, we write
\begin{align*}
&y^{k+2l-\alpha-\beta}\|D^{k+2l}_y U(\cdot,y)\|_{L_\infty}=
y^{k+2l-\alpha-\beta}\|D^{k}_y (-\Delta)^l U(\cdot,y)\|_{L_\infty}\\
&=y^{k+2l-\alpha-\beta}\|D^{k}_y (-\Delta)^{\beta/2} U(\cdot,y/2)*(-\Delta)^{l-\beta/2}P(\cdot,y/2)\|_{L_\infty}\\
&\le y^{k+2l-\alpha-\beta}\|D^{k}_y (-\Delta)^{\beta/2}U(\cdot,y/2)\|_{L_\infty} \|(-\Delta)^{l-\beta/2}P(\cdot,y/2)\|_{L_1}\\
&\le Ny^{k-\alpha}\|D^{k}_y (-\Delta)^{\beta/2}U(\cdot,y/2)\|_{L_\infty},
\end{align*}
which is less than the right-hand side of \eqref{eq22.22}. This completes the proof of the lemma.
\end{proof}

\begin{lemma}
                                    \label{lem3.5}
Let $\alpha>0$ and $\sigma\in (0,2)$ be constants. Assume that
$$
|K(y)|\le (2-\sigma)\Lambda |y|^{-d-\sigma}.
$$
Then  $L$ defined in \eqref{eq23.22.24} is a continuous operator from $\Lambda^{\alpha+\sigma}$ to $\Lambda^{\alpha}$, and we have
$$
\|Lu\|_{\Lambda^{\alpha}}\le N\Lambda\|u\|_{\Lambda^{\sigma+\alpha}},
$$
where $N=N(d,\sigma,\alpha)>0$ is uniformly bounded as $\sigma\to 2$.
\end{lemma}
\begin{proof}
Following the first part of the proof of the previous lemma, it suffices to show
$$
\|LP(\cdot,y/2)\|_{L_1}\le N(d,\sigma)\Lambda y^{-\sigma},\quad\forall y>0.
$$
By a dilation, we only need to verify that
\begin{equation}
                                        \label{eq14.19}
\|LP(\cdot,1)\|_{L_1}\le N(d,\sigma)\Lambda.
\end{equation}
Recall that
$$
P(x,1)=C_d(1+|x|^2)^{-(d+1)/2}:=p(x).
$$
Because the symbol of $L$ is given by \eqref{eq1102}, $m(\xi)\le N|\xi|^\sigma$ and
$$
LP(x,1)=\cF^{-1}(m(\xi)e^{-|\xi|}),
$$
it is easily seen that
\begin{equation}
                            \label{eq15.30}
\|LP(\cdot,1)\|_{L_\infty}\le N(d,\sigma)\Lambda.
\end{equation}
We will prove a pointwise decay estimate of $LP(\cdot,1)$:
\begin{equation}
                                    \label{eq16.20}
|LP(x,1)|\le N(d,\sigma)\Lambda|x|^{-(d+\sigma)},\quad\forall x\in B_1^c,
\end{equation}
which together with \eqref{eq15.30} yields \eqref{eq14.19}.
We fix a $x\in B_1^c$ and consider three cases: $\sigma<1$, $\sigma>1$, and $\sigma=1$.

{\em Case 1: $\sigma<1$.} In this case,
\begin{equation}
                                        \label{eq14.50}
|LP(x,1)|\le N(2-\sigma)\Lambda\int_{\bR^d}M_1(x,y)|y|^{-d-\sigma}\,dy,
\end{equation}
where
$$
M_1(x,y)=\big|p(x+y)-p(x)\big|.
$$
We divide $\bR^d$ into three regions. For any $y\in B_{|x|/2}$, we have
\begin{equation}
                                    \label{eq16.14}
M_1(x,y)\le N(d)|y|(1+|x|^2)^{-(d+2)/2}.
\end{equation}
Thus
\begin{align*}
\int_{B_{|x|/2}}M_1(x,y)|y|^{-d-\sigma}\,dy
&\le N(d)(1+|x|^2)^{-(d+2)/2}
\int_{B_{|x|/2}}|y|^{1-d-\sigma}\,dy\nonumber\\
&\le N(d)(1-\sigma)^{-1}|x|^{-(d+1+\sigma)}.
\end{align*}
For any $y\in B_{|x|/2}(-x)$, we have $|y|\in (|x|/2,3|x|/2)$,
$$
M_1(x,y)\le N(d)(1+|x+y|^2)^{-(d+1)/2},
$$
which implies
\begin{multline*}
\int_{B_{|x|/2}(-x)}M_1(x,y)|y|^{-d-\sigma}\,dy
\le N(d)|x|^{-d-\sigma}
\int_{B_{|x|/2}(-x)}(1+|x+y|^2)^{-(d+1)/2}\,dy\\
\le N(d)|x|^{-d-\sigma}
\int_{B_{|x|/2}}(1+|y|^2)^{-(d+1)/2}\,dy \le N(d)|x|^{-(d+\sigma)}.
\end{multline*}
For any $y\in B_{|x|/2}^c\cap B^c_{|x|/2}(-x):=\Omega$, we have
$$
M_1(x,y)\le N(d)(1+|x|^2)^{-(d+1)/2}.
$$
Therefore,
\begin{align*}
\int_{\Omega}M_1(x,y)|y|^{-d-\sigma}\,dy
&\le N(d)\int_{B_{|x|/2}^c}|y|^{-d-\sigma}(1+|x|^2)^{-(d+1)/2}\,dy\nonumber\\
&\le N(d)\sigma^{-1}|x|^{-(d+1+\sigma)}.
\end{align*}
Combining the estimates above with \eqref{eq14.50}, we obtain \eqref{eq16.20}.

{\em Case 2: $\sigma>1$.} In this case, we have \eqref{eq14.50} with $M_1$ replaced by
$$
M_2(x,y)=\big|p(x+y)-p(x)-y\cdot \nabla p(x)\big|.
$$
For any $y\in B_{|x|/2}$, we have
$$
M_2(x,y)\le N|y|^2(1+|x|^2)^{-(d+3)/2}.
$$
Hence
\begin{align*}
\int_{B_{|x|/2}}M_2(x,y)|y|^{-d-\sigma}\,dy
&\le N(1+|x|^2)^{-(d+3)/2}
\int_{B_{|x|/2}}|y|^{2-d-\sigma}\,dy\nonumber\\
&\le N(d)(2-\sigma)^{-1}|x|^{-(d+1+\sigma)}.
\end{align*}
In $B_{|x|/2}(-x)$ we estimates as in Case 1. In $\Omega$, the estimate \eqref{eq16.14} holds with $M_2$ in place of $M_1$. Therefore,
\begin{align*}
\int_{\Omega}M_2(x,y)|y|^{-d-\sigma}\,dy
&\le N\int_{B_{|x|/2}^c}|y|^{1-d-\sigma}(1+|x|^2)^{-(d+2)/2}\,dy\nonumber\\
&\le N|x|^{-(d+1+\sigma)}.
\end{align*}
Thus we obtain \eqref{eq16.20} in this case.

{\em Case 3: $\sigma=1$.} In this case, using the condition \eqref{eq21.47} we get
$$
|LP(x,1)|\le N(d)\Lambda\Big(\int_{B_{|x|/2}}M_2(x,y)|y|^{-d-\sigma}\,dy
+\Lambda\int_{B_{|x|/2}^c}M_1(x,y)|y|^{-d-\sigma}\,dy\Big).
$$
We then derive \eqref{eq16.20} by estimating the first integral on the right-hand side as in Case 2 and the second integral as in Case 1.

The lemma is proved.
\end{proof}

Now we give the proof of Theorem \ref{thm2}.

\begin{proof}[Proof of Theorem \ref{thm2}]
Thanks to Lemma \ref{lem3.5}, we have the continuity of $L-\lambda$ and \eqref{eq22.100}.
By Property (6), the estimate \eqref{eq17.24} follows once we prove it
for an $\alpha \in (0,1)$ and any $u \in \Lambda^{\sigma+\alpha}$.
In this case, \eqref{eq17.24} follows immediately from Theorem \ref{mainth01} i), the standard mollifications, and Lemma \ref{lem3.4}.
For the proof of \eqref{eq0316_01}, we use the proof of Theorem \ref{mainth01} and Lemma \ref{lem3.2}. In particular, Lemma \ref{lem3.2} allows us to replace the last term $[f]_{C^\alpha}$ in \eqref{eq11.34} by $[f]_{\Lambda^\alpha}$.
Therefore, it remains to show the unique solvability of the equation \eqref{elliptic}. With the a priori estimates in hand, the argument below is quite standard; see, for instance, \cite[Chap. 3 \& 4]{Kr96}.

Due to Property (6), without loss of generality we may assume $\alpha\in (0,1)$. We first consider the special case $L=-(-\Delta)^{\sigma/2}$. In this case, for any $\varepsilon>0$ let $f^{(\varepsilon)}$ be the standard mollification of $f$. As $f\in \Lambda^{\alpha}$, it is easily seen that $f^{(\varepsilon)}\to f$ in $\Lambda^\beta$ for any $\beta\in (0,\alpha)$. We fix a $\beta\in (0,\alpha)$. By Lemma \ref{lem3.3}, the equation
\begin{equation}
                                        \label{eq11.56}
-(-\Delta)^{\sigma/2}u^\varepsilon-\lambda u^\varepsilon=f^{(\varepsilon)}
\end{equation}
has a solution $u^\varepsilon=G_\sigma*f^{(\varepsilon)}$ in $C^\infty$ with bounded derivatives. Due to the estimate \eqref{eq17.24}, we have
\begin{equation}
                                        \label{eq11.57}
\|u^\varepsilon\|_{\Lambda^{\sigma+\alpha}}\le N\|f^{(\varepsilon)}\|_{\Lambda^{\alpha}}
\le N\|f\|_{\Lambda^{\alpha}},
\end{equation}
where $N=N(d,\nu,\Lambda,\sigma,\alpha,\lambda)$, and
$$
\|u^{\varepsilon}-u^{\tilde \varepsilon}\|_{\Lambda^{\sigma+\beta}}\le N\|f^{(\varepsilon)}-f^{(\tilde \varepsilon)}\|_{\Lambda^{\beta}}\to 0\quad \text{as}\,\,\varepsilon,\tilde \varepsilon\to 0.
$$
Therefore, there exists a function $u\in \Lambda^{\sigma+\beta}$ such that $u^{\varepsilon}\to u$ in $\Lambda^{\sigma+\beta}$. Passing to the limit in \eqref{eq11.56} and \eqref{eq11.57} and using Lemma \ref{lem3.4}, we infer that $u$ is a solution to
\begin{equation*}
-(-\Delta)^{\sigma/2}u-\lambda u=f
\end{equation*}
and it satisfies \eqref{eq17.24}. The general case is then derived from the special case by using \eqref{eq17.24} and the standard method of continuity together with the continuity estimate in Lemma \ref{lem3.5}. Finally, the uniqueness is a direct consequence of the $L_\infty$ estimate \eqref{eq07.09.53}. This completes the proof of the theorem.
\end{proof}

\section{Local estimates and $x$-dependent kernels}
                                                \label{sec5}

In this section we assume that $0 < \sigma < 2$ and $\alpha \in (0,1)$. We first derive several local estimates, i.e., Corollary \ref{thm3}, and then give an outline of the proof of Schauder estimates for operators with $x$-dependent kernels, i.e., Theorem \ref{thm4}.

\begin{proof}[Proof of Corollary \ref{thm3}]
We take a cut-off function $\eta\in
C_0^\infty(\bR^d)$ with a compact support in $B_2$ satisfying $\eta\equiv 1$ on $B_1$. Then it is
easily seen that
$$
L(\eta u)-\lambda \eta u=\eta f+L(\eta u)-\eta Lu
\quad
\text{in}\,\,\bR^d.
$$
Applying the global estimate in Theorem \ref{mainth01} to the equation above gives
$$
[(-\Delta)^{\sigma/2}(\eta u)]_{C^\alpha(B_1)}\le N[\eta f+L(\eta u)-\eta Lu]_{C^\alpha(\bR^d)}
$$
$$
\le N[\eta f]_{C^\alpha(\bR^d)}+N[L(\eta u)-\eta Lu]_{C^\alpha(\bR^d)}.
$$
Thus, by the triangle inequality,
\begin{multline}
                        \label{eq21.39}
[(-\Delta)^{\sigma/2}u]_{C^\alpha(B_1)}\le
[ \eta (-\Delta)^{\sigma/2}u - (-\Delta)^{\sigma/2}(\eta u) ]_{C^\alpha(B_1)}
\\
+ [(-\Delta)^{\sigma/2}(\eta u)]_{C^\alpha(B_1)}
\le [ \eta (-\Delta)^{\sigma/2}u - (-\Delta)^{\sigma/2}(\eta u) ]_{C^\alpha(B_1)}
\\
+N[\eta f]_{C^\alpha(\bR^d)}+N[L(\eta u)-\eta Lu]_{C^\alpha(\bR^d)}.
\end{multline}
Note that
\begin{equation}
							\label{eq1222}
[\eta f]_{C^\alpha(\bR^d)} \le [f]_{C^\alpha(B_3)}
+ N \|f\|_{L_\infty(B_3)}.
\end{equation}
Now we estimate the third term on the right-hand side of \eqref{eq21.39}.
The estimate of the term $[ \eta (-\Delta)^{\sigma/2}u - (-\Delta)^{\sigma/2}(\eta u) ]_{C^\alpha(B_1)}$ is similarly obtained by using $(-\Delta)^{\sigma/2}$ in place of $L$. Let
\begin{align}
							\label{eq1223_01}
h(x) : &= L(\eta u)-\eta Lu\nonumber\\
&=\int_{\bR^d}\Big(\big(\eta(x+y)-\eta(x)\big)u(x+y)-y\cdot \nabla \eta(x) u(x)\chi^{(\sigma)}(y)\Big)K(y)\,dy.
\end{align}
Throughout the proof, we set
\begin{equation}
							\label{eq1223}
\delta_y g(x) = g(x+y) - g(x),
\quad
\delta_y^2 g(x) = g(x+y) - g(x) - y \cdot \nabla g(x).
\end{equation}

{\em (i)} For $\sigma\in (0,1)$,
from \eqref{eq1223_01} and \eqref{eq1223} we have
$$
\frac{h(x_0) - h(x_1)}{|x_0 - x_1|^\alpha}
= \int_{\bR^d} \delta_y \eta (x_0)\frac{u(x_0+y)-u(x_1+y)}{|x_0 - x_1|^\alpha} K(y) \, d y
$$
$$
+ \int_{\bR^d} \frac{\delta_y \eta (x_0) - \delta_y \eta (x_1)}{|x_0 - x_1|^\alpha} u(x_1+y) K(y) \, dy := I_1 + I_2,
$$
where $x_0, x_1 \in \bR^d$.
Observe that
$$
|I_1| \le \int_{\bR^d} \left|\delta_y \eta (x_0)\right| \frac{\left| u(x_0 + y) - u(x_1 + y) \right|}{|x_0 - x_1|^\alpha} K(y) \, dy
\le N [u]_{C^\alpha(\bR^d)},
$$
and
\begin{align*}
|I_2|
&\le \int_{\bR^d} \left( [\eta]_{C^\alpha(\bR^d)} 1_{|y|\ge 1} + [\nabla\eta]_{C^\alpha(\bR^d)} |y| 1_{|y|<1}\right) |u(x_1+y)| K(y) \, dy\\
&\le N \|u\|_{L_\infty(\bR^d)}.
\end{align*}
Thus
\begin{equation}
							\label{eq1223_02}
[L(\eta u)-\eta Lu]_{C^\alpha(\bR^d)}
\le  N \|u\|_{C^\alpha(\bR^d)}.
\end{equation}

{\em (ii)} For $\sigma\in (1,2)$, we write
$$
h(x) = \int_{\bR^d} \delta_y \eta(x) \delta_y u(x) K(y) \, dy
+ \int_{\bR^d} \delta_y^2 \eta(x) u(x) K(y) \, dy.
$$
Then, for $x_0, x_1 \in \bR^d$,
\begin{align*}
\frac{h(x_0) - h(x_1)}{|x_0-x_1|^\alpha}
=& \int_{\bR^d} \delta_y \eta(x_0) \frac{\delta_y u(x_0) - \delta_y u(x_1)}{|x_0-x_1|^\alpha} K(y) \, dy\\
&+ \int_{\bR^d} \frac{\delta_y \eta(x_0) - \delta_y \eta(x_1)}{|x_0-x_1|^\alpha} \delta_y u(x_1) K(y) \, dy\\
&+ \int_{\bR^d} \delta_y^2 \eta(x_0) \frac{u(x_0) - u(x_1)}{|x_0-x_1|^\alpha} K(y) \, dy\\
&+ \int_{\bR^d} \frac{\delta_y^2 \eta(x_0) - \delta_y^2 \eta(x_1)}{|x_0-x_1|^\alpha} u(x_1) K(y) \, dy
:= I_3 + I_4 + I_5 + I_6.
\end{align*}
Observe that
$$
|I_3|
\le \int_{\bR^d} |\eta(x_0+y)-\eta(x_0)| [\nabla u]_{C^\alpha(\bR^d)} |y| K(y) \, dy
\le N [\nabla u]_{C^\alpha(\bR^d)},
$$
$$
|I_4|
\le
\int_{\bR^d} [\nabla \eta]_{C^\alpha(\bR^d)} |y| | \delta_y u(x_1) | K(y) \, dy
\le N \|u\|_{L_\infty(\bR^d)} + N \|\nabla u\|_{L_\infty(\bR^d)},
$$
$$
|I_5|
\le N [u]_{C^\alpha(\bR^d)},
\quad
|I_6|
\le N \|u\|_{L_\infty(\bR^d)}.
$$
Hence
\begin{equation}
						\label{eq1223_03}
[L(\eta u)-\eta Lu]_{C^\alpha(\bR^d)}
\le  N \|u\|_{C^{1+\alpha}(\bR^d)}.
\end{equation}

{\em (iii)} In the last case $\sigma=1$, by using \eqref{eq21.47},
for any $\delta\in (0,1)$ we write
\begin{align*}
h(x) = &\int_{B_\delta} \delta_y\eta(x) \delta_y u(x) K(y) \, dy
+ \int_{B_\delta} \delta_y^2 \eta(x) u(x) K(y) \, dy\\
&+ \int_{B_\delta^c} \delta_y \eta(x) u(x+y) K(y) \, dy.
\end{align*}
Then, for $x_0, x_1 \in \bR^d$,
\begin{align*}
\frac{h(x_0) - h(x_1)}{|x_0-x_1|^\alpha}
=& \int_{B_\delta} \delta_y \eta(x_0) \frac{\delta_y u(x_0) - \delta_y u(x_1)}{|x_0-x_1|^\alpha} K(y) \, dy\\
&+ \int_{B_\delta} \frac{\delta_y \eta(x_0) - \delta_y \eta(x_1)}{|x_0-x_1|^\alpha} \delta_y u(x_1) K(y) \, dy\\
&+ \int_{B_\delta} \delta_y^2 \eta(x_0) \frac{u(x_0) - u(x_1)}{|x_0-x_1|^\alpha} K(y) \, dy\\
&+ \int_{B_\delta} \frac{\delta_y^2 \eta(x_0) - \delta_y^2 \eta(x_1)}{|x_0-x_1|^\alpha} u(x_1) K(y) \, dy\\
&+ \int_{B_\delta^c} \delta_y \eta(x_0) \frac{ u(x_0+y) - u(x_1+y) }{|x_0 - x_1|^\alpha} K(y) \, dy\\
&+ \int_{B_\delta^c} \frac{ \delta_y \eta(x_0) - \delta_y \eta(x_1) }{|x_0 - x_1|^\alpha} u(x_1+y) K(y) \, dy
:= \sum_{i=7}^{12}I_i.
\end{align*}
We see that
\begin{align*}
&|I_7| \le N \delta [\nabla u]_{C^\alpha(\bR^d)},\quad
|I_8| \le N \delta \|Du\|_{L_\infty(\bR^d)},\\
&|I_9| \le N \delta [u]_{C^\alpha(\bR^d)},
\quad
|I_{10}| \le N \delta \|u\|_{L_\infty(\bR^d)},\\
&|I_{11}| \le N \delta^{-1} [u]_{C^\alpha(\bR^d)},
\quad
|I_{12}| \le N \delta^{-1} \|u\|_{L_\infty(\bR^d)}.
\end{align*}
Therefore, upon choosing an appropriate $\delta \in (0,1)$ and using an interpolation inequality, we have
\begin{equation}
							\label{eq1223_04}
[L(\eta u)-\eta Lu]_{C^\alpha(\bR^d)}
\le  \varepsilon \|u\|_{C^{1+\alpha}(\bR^d)} + N(\varepsilon) \|u\|_{L_\infty(\bR^d)}.
\end{equation}

The inequalities \eqref{eq1223_02}, \eqref{eq1223_03}, and \eqref{eq1223_04},
as well as the same inequalities with $(-\Delta)^{\sigma/2}$ in place of $L$ along with \eqref{eq21.39} and \eqref{eq1222}
give the desired inequalities \eqref{eq21.06}, \eqref{eq21.06b}, and \eqref{eq21.06c}.
\end{proof}

We recall that the classical Schauder theory for second-order equations with variable coefficients is built upon the estimates for equations with constant coefficients by using a standard perturbation argument and a localization technique; see, for instance, \cite{Kr96}. This method works equally well for the nonlocal operator \eqref{eq23.22.24} when $a(x,y)$ is $C^\alpha$ H\"{o}lder continuous with respect to $x$. However, the corresponding argument is a bit more lengthy. See, for instance,  the perturbation and localization arguments in the proof of Theorem 1.2 in \cite{Bas09}. Here we give a sketched proof.

We first derive the following corollary of Lemma \ref{lem3.5}.

\begin{corollary}
                                                \label{cor3.5}
Let $\sigma\in (0,2)$ and $\alpha\in (0,1)$ be constants. Assume that $K(x,y)=a(x,y)|y|^{-d-\sigma}$, where $a(\cdot,y)$ is a $C^\alpha$ function for any $y\in \bR^d$ with a uniform $C^\alpha$ norm. In the case $\sigma=1$, the condition \eqref{eq21.47} is also satisfied for any $x\in \bR^d$. Then  $L$ defined in \eqref{eq23.22.24} is a continuous operator from $\Lambda^{\alpha+\sigma}$ to $\Lambda^{\alpha}$, and for any $\beta\in (0,\alpha)$ we have
$$
\|Lu\|_{\Lambda^{\alpha}}\le N_1\sup_y\|a(\cdot,y)\|_{L_\infty}\|u\|_{\Lambda^{\sigma+\alpha}}
+N_2\sup_y[a(\cdot,y)]_{C^\alpha}\|u\|_{\Lambda^{\sigma+\beta}},
$$
where $N_1=N_1(d,\sigma,\alpha)>0$ and $N_2=N_1(d,\sigma,\alpha,\beta)>0$ are constants.
\end{corollary}
\begin{proof}
Without loss of generality, we may assume that $\beta$ is small such that $\sigma+\beta\in (0,1)$ when $\sigma\in (0,1)$ and $\sigma+\beta\in (1,2)$ when $\sigma\in [1,2)$. By a translation of the coordinates, it suffices to show that
\begin{equation}
                            \label{eq15.28}
|Lu(0)|\le N_1\sup_y\|a(\cdot,y)\|_{L_\infty}\|u\|_{\Lambda^{\sigma+\alpha}},
\end{equation}
and for any $x\in\bR^d$,
\begin{multline}
                                        \label{eq15.29}
|Lu(x)-Lu(0)|\\
\le \big(N_1\sup_y\|a(\cdot,y)\|_{L_\infty}\|u\|_{\Lambda^{\sigma+\alpha}}
+N_2\sup_y[a(\cdot,y)]_{C^\alpha}\|u\|_{\Lambda^{\sigma+\beta}}\big)|x|^\alpha.
\end{multline}
Let $L_0$ be the operator with the kernel $K(x,y)$ in $L$ replaced by $K(0,y)$.
Then the inequality \eqref{eq15.28} follows from Lemma \ref{lem3.5} and the obvious identity $Lu(0)=L_0u(0)$. By the triangle inequality,
\begin{equation}
                                \label{eq15.35}
|Lu(x)-Lu(0)|\le |Lu(x)-L_0u(x)|+|L_0u(x)-L_0u(0)|.
\end{equation}
We estimate the second term on the right-hand side of \eqref{eq15.35} by Lemma \ref{lem3.5},
\begin{align}
                                        \label{eq15.36}
|L_0u(x)-L_0u(0)|&\le N_1|x|^\alpha\sup_{y\in \bR^d} |a(0,y)|\|u\|_{\Lambda^{\sigma+\alpha}}\nonumber\\
&\le N_1|x|^\alpha\sup_y\|a(\cdot,y)\|_{L_\infty}\|u\|_{\Lambda^{\sigma+\alpha}}.
\end{align}
To estimate the first term on the right-hand side of  \eqref{eq15.35}, we compute
\begin{align*}
&|Lu(x)-L_0u(x)|\\
&\le \int_{\bR^d} \big|u(x+y) - u(x)-y\cdot \nabla u(x)\chi^{(\sigma)}(y)\big||a(x,y)-a(0,y)||y|^{-d-\sigma}\, dy\\
&\le \sup_y[a(\cdot,y)]_{C^\alpha}|x|^\alpha
\int_{\bR^d} \big|u(x+y) - u(x)-y\cdot \nabla u(x)\chi^{(\sigma)}(y)\big||y|^{-d-\sigma}\, dy.
\end{align*}
By Taylor's formula, we have for any $y\in B_1$,
$$
\big|u(x+y) - u(x)-y\cdot \nabla u(x)\chi^{(\sigma)}(y)\big|\le N|y|^{\sigma+\beta}\|u\|_{\Lambda^{\sigma+\beta}},
$$
and for any $y\in B_1^c$,
$$
\big|u(x+y)-u(x)-y\cdot \nabla u(x)\chi^{(\sigma)}(y)\big|\le N|y|^{s}\|D^su\|_{L_\infty},
$$
where $s=0$ when $\sigma\in (0,1]$ and $s=1$ when $\sigma\in (1,2)$. Therefore, we get
\begin{equation}
                                \label{eq16.09}
|Lu(x)-L_0u(x)|\le \sup_y[a(\cdot,y)]_{C^\alpha}|x|^\alpha\|u\|_{\Lambda^{\sigma+\beta}}.
\end{equation}
Collecting \eqref{eq15.35}, \eqref{eq15.36}, and \eqref{eq16.09}, we reach \eqref{eq15.29}. The corollary is proved.
\end{proof}

Now we are ready to complete the proof of Theorem \ref{thm4}
\begin{proof}[Proof of Theorem \ref{thm4}]
We first derive an a priori estimate for $[(-\Delta)^{\sigma/2} u]_{C^\alpha}$. Let
$$
M:=\sup_{x\in B_\delta}|(-\Delta)^{\sigma/2} u(x)-(-\Delta)^{\sigma/2} u(0)|/|x|^\alpha
$$
for some small constant $\delta>0$ to be chosen later. Rewrite \eqref{elliptic} as
$$
L_0 u-\lambda u=\tilde f:=f+(L_0-L)u,
$$
where $L_0$ is defined as in \eqref{eq23.22.24} with $K(0,y)$ in place of $K(x,y)$. Now by the local estimates in Corollary \ref{thm3} together with a scaling, we have
\begin{equation*}
M\le
N\|\tilde f\|_{C^\alpha(B_{3\delta})}+N(\delta)\|u\|_{C^\alpha(\bR^d)}
\end{equation*}
for $\sigma\in (0,1)$,
\begin{equation*}
M\le
N\|\tilde f\|_{C^\alpha(B_{3\delta})}+N(\delta)\|u\|_{C^{1+\alpha}(\bR^d)}
\end{equation*}
for $\sigma\in (1,2)$, and
\begin{equation}
                                \label{eq21.06cc}
M\le N\|\tilde f\|_{C^\alpha(B_{3\delta})}
+ N(\varepsilon, \delta)\|u\|_{L_\infty(\bR^d)}+\varepsilon\|u\|_{C^{1+\alpha}(\bR^d)}
\end{equation}
for $\sigma=1$ and any $\varepsilon\in (0,1)$. It remains to bound the $C^\alpha$ norm of
\begin{align*}
&\eta_{3\delta}(L_0-L)u(x)\\
&\,=\int_{\bR^d} \left(u(x+y) - u(x)-y\cdot \nabla u(x)\chi^{(\sigma)}(y)\right)|y|^{-d-\sigma}\tilde a(x,y)\, d y,
\end{align*}
where $\tilde a(x,y)=\eta_{3\delta}(x)\big(a(0,y)-a(x,y)\big)$, $\eta_{3\delta}(\cdot)=\eta(\cdot/(3\delta))$, and $\eta$ is the cut-off function in the proof of Corollary \ref{thm3}.
Using Corollary \ref{cor3.5}, one can bound the $C^\alpha$ norm of the above integral by a lower-order norm and
$$
N\sup_y\|\tilde a(\cdot,y)\|_{L_\infty}\|u\|_{\Lambda^{\sigma+\alpha}}\le N\delta^\alpha \|u\|_{\Lambda^{\sigma+\alpha}}.
$$
By a translation of coordinates and interpolation inequalities, we then reach
\begin{equation}
                        \label{eq6.06}
\|u\|_{\lambda^{\sigma+\alpha}}\le N\delta^\alpha\|u\|_{\lambda^{\sigma+\alpha}}+N\|f\|_{C^\alpha}+N(\delta) \|u\|_{L_\infty}.
\end{equation}
Here we have chosen $\varepsilon$ sufficiently small in \eqref{eq21.06cc} when $\sigma=1$.
By the maximum principle, it holds that
\begin{equation}
                                \label{eq6.15}
\lambda\|u\|_{L_\infty}\le \|f\|_{L_\infty}.
\end{equation}
Combining \eqref{eq6.06} and \eqref{eq6.15}, and taking $\delta$ sufficiently small, we obtain an apriori estimate
\begin{equation*}
\|u\|_{\lambda^{\sigma+\alpha}}\le N\|f\|_{C^\alpha}.
\end{equation*}
The continuity of $L$ has already been proved in Corollary \ref{cor3.5}.
The unique solvability then follows immediately by using the standard method of continuity argument. This completes the proof of the theorem.
\end{proof}

\section*{Acknowledgement}
The authors are grateful to Richard Bass for kindly informing us the paper \cite{Bas09} and to Nicolai V. Krylov for helpful comments. They would also like to thank the referee for many useful suggestions.

\bibliographystyle{plain}

\end{document}